\documentclass{amsart}

\usepackage{graphicx, color}
\usepackage{hyperref}
\usepackage{enumerate}
\usepackage[all]{xy}
\usepackage{amsmath,amsfonts,amssymb}
\usepackage[latin9]{inputenc}

\newcommand{\Z}{{\mathbb Z}}

\newcommand{\VV}{{\mathbb V}}
\newcommand{\WW}{{\mathbb W}}

\newcommand{\ZZ}{{\mathbb Z}}
\newcommand{\MM}{{\mathcal M}}
\newcommand{\TT}{{\mathbb T}}
\newcommand{\DD}{{\mathbb D}}

\def\ol{\overline}

\newtheorem{prop}{Proposition}[section]
 
\newtheorem{corollary}[prop]{Corollary}

\newtheorem{lemma}[prop]{Lemma}
\newtheorem{theorem}[prop]{Theorem}
\newtheorem{proposition}[prop]{Proposition}
\newtheorem{conjecture}[prop]{Conjecture}

\theoremstyle{definition}

\newtheorem{defi}[prop]{Definition}
\newtheorem{exmp}[prop]{Example}

\theoremstyle{remark}

\newtheorem{remark}[prop]{Remark}
\newtheorem{remarks}[prop]{Remarks}

\newcommand{\benu}{\begin{enumerate}}
\newcommand{\enu}{\end{enumerate}}

\newcommand{\beqna}{\begin{eqnarray}}
\newcommand{\eqna}{\end{eqnarray}}
\newcommand{\beqnast}{\begin{eqnarray*}}
\newcommand{\eqnast}{\end{eqnarray*}}
\newcommand{\beqn}{\begin{equation}}
\newcommand{\eqn}{\end{equation}}
\newcommand{\beqnst}{\begin{equation*}}
\newcommand{\eqnst}{\end{equation*}}

\usepackage{latexsym}

\usepackage{xspace}
\usepackage{amscd}
\usepackage{amssymb}
\usepackage{amsfonts}

\makeatother

\newcommand{\bema}{\left ( \begin{array}}
\newcommand{\ema}{\end{array} \right )}

\newcommand{\Hom}{\operatorname{Hom}}
\newcommand{\End}{\operatorname{End}}

\newcommand{\ot}{\otimes}

\def\Aa{{\mathcal A}}

\def\ul{\underline}
\newcommand{\thlabel}[1]{\label{th:#1}}
\newcommand{\thref}[1]{Theorem~\ref{th:#1}}
\newcommand{\selabel}[1]{\label{se:#1}}
\newcommand{\seref}[1]{Section~\ref{se:#1}}
\newcommand{\lelabel}[1]{\label{le:#1}}
\newcommand{\leref}[1]{Lemma~\ref{le:#1}}
\newcommand{\prlabel}[1]{\label{pr:#1}}
\newcommand{\prref}[1]{Proposition~\ref{pr:#1}}

\newcommand{\relabel}[1]{\label{re:#1}}

\newcommand{\exlabel}[1]{\label{ex:#1}}
\newcommand{\exref}[1]{Example~\ref{ex:#1}}

\newcommand{\eqlabel}[1]{\label{eq:#1}}
\newcommand{\equref}[1]{(\ref{eq:#1})}
\def\Cc{{\mathcal C}}

\def\Mm{{\mathcal M}}

\def\Hh{{\mathcal H}}

\begin{document}

\title[Dilations of partial representations of Hopf algebras]{Dilations of partial representations of Hopf algebras}

\author[M.M.S.\ Alves]{Marcelo \ Muniz \ S. \ Alves}
\address{
   Marcelo \ Muniz \ S. \ Alves\\
Departamento de Matem\'atica,\\ Universidade Federal do Paran\'a,\\ Brazil   }
   \email{marcelomsa@ufpr.br}
\author[E.\ Batista]{Eliezer Batista}
\address{
   Eliezer Batista\\
   Departamento de Matem\'atica,\\ Universidade Federal de Santa Catarina,\\ Brazil}
   \email{ebatista@mtm.ufsc.br}
\author[J.\ Vercruysse]{Joost Vercruysse}
\address{
   Joost Vercruysse\\
   D\'epartement de Math\'ematiques,\\ Universit\'e Libre de Bruxelles,\\ Belgium}
   \email{jvercruy@ulb.ac.be}


\begin{abstract} We introduce the notion of a dilation for a partial representation (i.e. a partial module) of a Hopf algebra, which in case the partial representation origins from a partial action (i.e.\ a partial module algebra) coincides with the enveloping action (or globalization). 
This construction leads to categorical equivalences between the category of partial $H$-modules, a category of (global) $H$-modules endowed with a projection satisfying a suitable commutation relation and the category of modules over a (global) smash product constructed upon $H$, from which we deduce the structure of a Hopfish algebra on this smash product. These equivalences are used to study the interactions between partial and global representation theory. 
\end{abstract}

\maketitle

\section{Introduction}
The importance of symmetry is undeniable in all areas of Mathematics. Symmetries are useful tools to reduce the complexity of systems and to find unexpected correlations between the studied objects. 
However, in many situations, one is only interested in a small portion of a larger object of which the symmetries are well-described, but where these global symmetries do not necessarily restrict to the considered sub-object. In such a case, it 
is useful to describe, at least partially, what remains from that global symmetry. This question has lead to the theory of partial actions. The first appearance of partial group actions occurred in the context of operator algebras \cite{E1}, in which the author investigated $\mathbb{Z}$-graded $C^*$-algebras which were not described as a usual crossed product. The theory of partial actions soon became important in the theory of dynamical systems \cite{ELQ} and, after the seminal paper \cite{DE}, attracted the attention of pure algebraists as well. For an overview of the historical development of the theory of partial actions see, for example \cite{B} and \cite{D}.

Since natural examples of partial actions result from restricting global actions, the question arises if all partial actions emerge this way, i.e.\ whether it is possible to construct a ``globalization'' for any given partial action. 
More precisely, given a partial action of a group $G$ on a set $X$ is it possible to find a global action of the same group $G$ on another set $Y$, containing $X$, such that the original partial action can be recovered by restricting the global action on $Y$ to a partial action on $X$ ? In the case of partial group actions on sets, this question was positively answered in \cite{Abadie1}. In more involved settings, some additional complications can appear. For example, when $X$ is considered as a Hausdorff topological space, in general the space which carries the globalized action is no longer Hausdorff. On the algebraic side, several globalization theorems for partial group actions were proved. In \cite{DE} the globalization of a partial action of a group $G$ over a unital algebra $A$ is achieved if, and only if, the partial domains $A_g$, for each $g\in G$, are unital ideals, that is, they are ideals of $A$ generated by a central idempotent $1_g$. In \cite{DES} the globalization for twisted partial actions was obtained. In \cite{ADES}, it was proved that every partial action of a group $G$ over an algebra $A$, such that the ideals $A_g$ are idempotent, is Morita equivalent to a globalizable partial action. 
The globalization theorem also enabled further developments: for instance, the Galois theory for partial group actions, introduced in \cite{DFP}, strongly relies on the globalization theorem.

Partial (co)actions of Hopf algebras came on the scene in \cite{CJ} and they were motivated first by an attempt to describe Galois theory for partial group actions into the framework of Galois corings. One of the first results obtained in the theory of partial actions of Hopf algebras was exactly the globalization theorem, which asserts that every partial action of a Hopf algebra $H$ on a unital algebra $A$ admits a globalization \cite{AB},  which however is not necessarily unital. The globalization theorem triggered several new results. For example in \cite{AB3} the authors obtained a version of Blattner-Montgomery theorem for the case of partial actions, generalizing a result in \cite{L} for the group case. Globalization theorems were also obtained in other contexts such as partial actions of Hopf algebras on $k$-linear categories \cite{AAB} and twisted partial actions of Hopf algebras \cite{ABDP}.

The aim of this present article is to extend the notion of globalization to the setting of partial $H$-modules. Let us first explain what these partial modules are.

A partial action of a Hopf algebra $H$ on a unital algebra $A$, is linear map $H\ot A\to A$ satisfying suitable coherence axioms that are a weakening of the classical axioms for an $H$-module algebra. However, 
where a (global) $H$-module algebra can be understood as an $H$-module that is at the same time an algebra such that both structures are compatible, the associativity axiom of a partial action $H\ot A\to A$
explicitly makes use of the algebra structure on $A$.  
The basic motivation of \cite{ABV} was to restore this situation by recovering a module property for partial actions, such that partial actions can be viewed as algebras in a suitable category of partial modules. For this purpose, the authors introduced the notion of a partial representation of a Hopf algebra, extending the existing notion for groups introduced in \cite{DE} and developed in \cite{DEP}. 
Given a Hopf algebra $H$, a universal Hopf algebroid $H_{par}$ was constructed in \cite{ABV}, with the property that the category of partial modules coincides with the (global) modules over $H_{par}$, ensuring a closed monoidal structure on the category of partial modules. 

Every ordinary $H$-module can be viewed as a partial $H$-module in a trivial way.  In view of the above globalization theorems, the natural question arises whether 
given a partial $H$-module $M$, is it possible to define in a canonical way an $H$-module $N$ such that the partial $H$-module structure on $M$ could be obtained by a restriction of the  $H$-module structure on $N$. This globalized module is what we call throughout this article a {\em dilation} of the partial representation of $H$ on the partial $H$-module $M$. The notion of a dilation for a partial representation of a group $G$ was first introduced by F. Abadie in \cite{Abadie2} (see also \cite{Abadie3} for some unpublished recent developments in the subject). In this present work, we define a new category $\mathbb{T}({}_H \mathcal{M})$, whose objects are $H$-modules which admit a projection satisfying a specific commutation relation with its $H$-module structure. We prove that there is a categorical equivalence between the category of partial $H$-modules and this category $\mathbb{T}({}_H \mathcal{M})$. Moreover, there is a monoidal structure on this category such that this equivalence is a strong monoidal functor. Finally, as the universal Hopf algebroid $H_{par}$ is known to be isomorphic to a partial smash product $A\# H$, which in turn is Morita equivalent with the smash product $\ol{A}\# H$ associated to the globalization of $A$. This Morita equivalence induces the structure of a Hopfish algebra on $\ol{A}\# H$.

This article is organized as follows: In section 2, we briefly recall the definitions and main results on partial representations of Hopf algebras as developed in \cite{ABV}. Moreover, we investigate the interaction between partial and global representation theory and illustrate the results by classifying in detail all partial representations of the unique two-dimensional Hopf algebra and Sweedler's four-dimensional Hopf algebra. In section 3, we introduce a new category $\mathbb{T}({}_H \mathcal{M})$ of $H$-modules with projections satisfying a specific commutation relation. We prove that for each object in the category $\mathbb{T}({}_H \mathcal{M})$ 
there is a canonical way to produce a partial $H$-module by using the projection; this defines a restriction functor, $R$, from the category  $\mathbb{T}({}_H \mathcal{M})$ into the category ${}_H\mathcal{M}^{par}$. In section 4, we introduce the concept of a 
dilation of a partial $H$-module and show that every partial $H$-module admits a standard dilation which is minimal and unique up to isomorphism. Moreover we show that the standard dilation provides a quasi-inverse for the restriction functor described above, and therefore the category $\mathbb{T}({}_H \mathcal{M})$ is equivalent to the category of partial $H$-representations.
Finally, in Section 5, we provide equivalent conditions for the dilation functor to be exact, which means that it is equivalent to the functor $\ol{H_{par}}\ot_{H_{par}}-$ and therefore, to know any dilation it is enough to understand the dilation of the universal Hopf algebroid $H_{par}$. Since we know moreover that $H_{par}$ can be described as a particular partial smash product $\ul{A\# H}$, we investigate the dilation of such smash products, which in turn is closely related to the globalization of the partial $H$-module algebra $A$. We finish by exploring the relation between dilations and the Morita context, constructed in \cite{AB}, which relates the category of partial modules with the category of modules over the (non-unital) global smash product $\ol A\# H$. Combining the results, it then follows that the algebra $\ol A\# H$ can be endowed with the structure of a Hopfish algebra.

Throughout the paper, all algebras and coalgebras are supposed to be over a fixed commutative field $k$ (for the basic constructions, it is enough to suppose that $k$ is a commutative ring). Unadorned tensor products are tensor products over $k$, and linear maps are $k$-linear maps. 
Given an algebra $A$ we  
denote the categories of left $A$-modules, right $A$-modules and $A$-bimodules respectively by 
${}_A \mathcal{M}$, $\Mm_A$ and ${}_A \mathcal{M}_A$.

\section{Global modules and partial modules}\selabel{globpart}

In this section we recall a few results from \cite{ABV} about partial representations of Hopf algebras. For more details and the relation with the group case we refer to \cite{ABV} or the review \cite{B} and the references therein.

A {\em partial action} of a Hopf algebra $H$ on a unital algebra $A$ is a linear map 
\[
\begin{array}{lccr} \cdot : & H\otimes A & \rightarrow & A \\
\, & h\otimes a & \mapsto & h\cdot a
\end{array}
\]
satisfying the following three axioms
\begin{enumerate}[({PA}1)]
\item $1_H \cdot a=a$, for every $a\in A$;
\item $h\cdot (ab)=(h_{(1)}\cdot a)(h_{(2)}\cdot b)$, for every $h\in H$ and $a,b \in A$;
\item $h\cdot (k\cdot a)=(h_{(1)}\cdot 1_A )(h_{(2)}k\cdot a)$, for every $h,k \in H$ and $a\in A$.
\end{enumerate}
A partial action is said to be symmetric if it also satisfies
\begin{enumerate}
\item[({PA}3')] $h\cdot (k\cdot a)=(h_{(1)}k\cdot a)(h_{(2)}\cdot 1_A )$, for every $h,k \in H$ and $a\in A$. 
\end{enumerate}
Throughout this text, we will assume that the partial actions are symmetric. A partial action is sometimes called a {\em partial $H$-module algebra}. Notice that $A$ is not an $H$-module but a partial $H$-module. To introduce this notion, let us first recall that a {\em partial representation} of a Hopf algebra $H$ is a unital algebra $B$ endowed with a linear map 
$\pi: H \rightarrow B$ such that 
\begin{enumerate}[{({PR}1)}]
\item $\pi (1_H)  =  1_B$; \label{partialrep1}
\item $\pi (h) \pi (k_{(1)}) \pi (S(k_{(2)}))  =   \pi (hk_{(1)}) \pi (S(k_{(2)})) $;
\label{partialrep2}
\item $\pi (h_{(1)}) \pi (S(h_{(2)})) \pi (k)  =   \pi (h_{(1)}) \pi (S(h_{(2)})k)$; \label{partialrep3} 
\item $\pi (h) \pi (S(k_{(1)})) \pi (k_{(2)}) = \pi (hS(k_{(1)})) \pi (k_{(2)})$; \label{partialrep4}
\item $\pi (S(h_{(1)}))\pi (h_{(2)}) \pi (k) = \pi (S(h_{(1)}))\pi (h_{(2)} k)$. \label{partialrep5}
\end{enumerate}
In case $H$ is cocommutative (e.g.\ a group algebra), then the last two axioms become redundant.

A (left) {\em partial module} over $H$ is a pair $(M,\pi)$, where $M$ is a $k$-vector space and $\pi:H\to \End_k(M)$ is a (left) partial representation of $H$. 
If $(M,\pi)$ and $(M',\pi')$ are partial $H$-modules, then a morphism of partial $H$-modules is a $k$-linear map $f:M\to M'$ satisfying $f\circ \pi(h)= \pi'(h)\circ f$ for all $h\in H$. 
The category which has partial $H$-modules as objects and morphisms between them as defined above is denoted by ${{}_H\mathcal{M}}^{par}$.

Out of a partial action $(A,\cdot)$, one can construct two examples of partial representations: the first is $(\End(A), h\mapsto h\cdot -)$ (which means exactly that $A$ is a partial $H$-module), the second is $(\ul{A\# H}, h\mapsto 1_A\# h)$ where $\ul{A\# H}$ denotes the partial smash product $$\underline{A\#H}=\{a\#h=a(h_{(1)}\cdot 1)\ot h_{(2)}|a\ot h\in A\ot H\} ,$$ which is a unital algebra with product $(a\# h)(b\# k)=a(h_{(1)}\cdot b)\# h_{(2)}k$. 

The study of partial representations and partial modules over a Hopf algebra $H$ is equivalent to the study of global representations and global modules over a new algebra $H_{par}$, which is called the {\em partial Hopf algebra of $H$}. (As we will explain below, $H_{par}$ is however not a Hopf algebra, but a Hopf algebroid.) Let us recall this construction. Consider the tensor algebra $T(H)$ over the underlying $k$-module of $H$, and denote by $[h]\in T(H)$ the element in the tensor algebra associated to any $h\in H$. Then $H_{par}$ is the quotient of $T(H)$ by the ideal generated by the following relations
\begin{enumerate}
\item $[1_H] = 1_{H_{par}}$; 
\item $[h][k_{(1)}][S(k_{(2)})] = [hk_{(1)}][S(k_{(2)})]$, for all $h,k \in H$;
\item $[h_{(1)}][S(h_{(2)})][k] = [h_{(1)}][S(h_{(2)})k]$, for all $h,k \in H$;
\item $[h][S(k_{(1)})][k_{(2)}] = [hS(k_{(1)})][k_{(2)}]$, for all $h,k \in H$;
\item $[S(h_{(1)})][h_{(2)}][k] = [S(h_{(1)})][h_{(2)}k]$, for all $h,k \in H$.
\end{enumerate}

\noindent
The next theorem collects the main results of \cite{ABV}. 

\begin{theorem}
Let $H$ be a Hopf algebra with invertible antipode, and $H_{par}$ the associated algebra as introduced above. Moreover, denote by $A_{par}$ the subalgebra of $H_{par}$ generated by the elements of the form $\varepsilon_h:=[h_{(1)}][S(h_{(2)})]$ for all $h\in H$.
Then the following statements hold.
\begin{enumerate}[(i)]
\item The map $\cdot:H\ot A_{par}\to A_{par}, h\cdot \varepsilon_k= \varepsilon_{h_{(1)}k}\varepsilon_{h_{(2)}}$ defines a partial action of $H$ on $A_{par}$;
\item $H_{par} \cong \underline{A_{par}\# H}$ and moreover $H_{par}$ has a structure of a Hopf algebroid over the base algebra $A_{par}$;
\item For any partial representation $\pi :H\rightarrow B$, there is a unique algebra morphism
$\tilde{\pi}: H_{par} \rightarrow B$ such that $\pi =\tilde{\pi}\circ [\underline{\; }]$;
\item A partial $H$-module $(M,\pi )$, can be endowed with an $H_{par}$-module structure by 
\[
[h^1]\ldots [h^n]\triangleright m =\pi (h^1)(\ldots \pi (h^n)( (m))\ldots ).
\]
Conversely, given a left $H_{par}$-module $M$, one obtains a partial representation $\pi :H\rightarrow \mbox{End}_k (M)$ by
\[
\pi (h) (m) =[h]\triangleright m ;
\]
\item The above correspondences between partial $H$-modules and (global) $H_{par}$-modules induces an isomorphism of categories between ${}_H \mathcal{M}^{par}$ and ${}_{H_{par}}\mathcal{M}$. Consequently:
\begin{enumerate}
\item The category ${}_H\mathcal{M}^{par}$ has the structure of a closed monoidal category and the forgetful functor $U:{}_H \mathcal{M}^{par} \rightarrow {}_{A_{par}} \mathcal{M}_{A_{par}}$ is a strict monoidal functor that preserves internal homs;
\item 
The forgetful functor $U:{}_H \mathcal{M}^{par} \rightarrow {}_{A_{par}} \mathcal{M}_{A_{par}}$ is representable, and represented by the algebra $H_{par}$;
\end{enumerate}
\item The category of algebra objects in the monoidal category $({}_H \mathcal{M}^{par} ,\otimes_{A_{par}} ,A_{par})$ coincides with the category of partial actions (i.e. partial module algebras) over $H$.
\end{enumerate}
\end{theorem}

Throughout this article we will assume that the Hopf algebra $H$ has an invertible antipode.

\begin{exmp} \label{examplekparG} For $H=kG$, the group algebra of a group $G$. The partial ``Hopf'' algebra $(kG)_{par}$ coincides with the 
construction of the partial group algebra $k_{par}G$ given in \cite{DE}. 
\end{exmp}

\begin{exmp}\exlabel{trivialmod} Every $H$-module can be viewed as a partial $H$-module, defining the partial representation as the algebra morphism $\pi :H\rightarrow \End_k (M)$ given by $\pi (h)(m)=h\triangleright m$. On the other hand, a partial $H$-module 
$(M, \pi )$ is an ordinary $H$-module if, and only if, $\pi (h_{(1)}) \pi (S(h_{(2)}))=\epsilon (h)\mbox{Id}_M$, for every $h\in H$.

More precisely, the category of $H$-modules can be seen as a subcategory of the category of partial $H$-modules by means of a canonical inclusion functor $i:{}_H \mathcal{M} \rightarrow {}_H \mathcal{M}^{par}$. 
If $M$ is a global $H$-module, then the action of $A_{par}$ on $i(M)=M$ is trivial: $\epsilon_h\cdot m=\epsilon(h)m$. Hence, given two global $H$-modules upon which we apply the inclusion functor, we find that $i(M)\ot_{A_{par}} i(N)=M\ot N$. However the inclusion functor is not strictly monoidal (however still monoidal), since the monoidal unit is not preserved, as the morphism $A_{par}\to k, \epsilon_h\mapsto \epsilon(h)$ is not bijective.

Obviously, if there is an epimorphism of partial $H$-modules $p:N\to M$, where $N$ is any (global) $H$-module and $M$ is any partial $H$-module, then $M$ is global. Henceforth, ${}_H\Mm$ is the full subcategory of ${}_H\Mm^{par}$ generated by $H$ (or any other generator of ${}_H\Mm$).
 
\end{exmp}

To further investigate the relation between partial and global modules, we introduce some new definitions.

\begin{defi}
Let $M$ be a partial $H$-module. The {\em global core} of $M$ is the biggest global submodule of $M$. The {\em global shadow} is the biggest global quotient of $M$.
A partial $H$-module $M$ is called {\em pure} if it does not contain any non-trivial global $H$-module.
\end{defi}

Clearly, $H$ is the global shadow of $H_{par}$. In general, we have the following result.

\begin{theorem}
Let $M$ be a partial $H$-module. Then the global core of $M$ is given by
$$c(M)=\{m\in M~|~ g\cdot (h\cdot m)=(gh)\cdot m, \forall h,g\in H\}$$
and $c$ is a right adjoint to the inclusion functor.
The global shadow of $M$ is given by
$$s(M)=M/<h\cdot (g\cdot m)-(gh)\cdot m>$$
and $s$ is a left adjoint to the inclusion functor.
\end{theorem}

\begin{proof}
First of all, it is easily verified that $c(M)$ and $s(M)$ are indeed the global core and shadow of $M$.

We know (see \cite{ABV}) that the map 
$$H_{par}\to H, [h]\mapsto h$$
is an algebra map. Hence, $H$ is both an $H$-$H_{par}$ bimodule and an $H_{par}$-$H$ bimodule.
Hence, we obtain adjoint pairs
$(H\ot_H -,{_{H_{par}}}\Hom(H,-))$ and $(H\ot_{H_{par}}-, {_H\Hom}(H,-))$ as follows
\[
\xymatrix{
{_H\Mm} \ar@<.5ex>[rr]^-{H\ot_H-} && {_{H_{par}}\Mm} \ar@<.5ex>[ll]^-{{_{H_{par}}\Hom}(H,-)}
}\quad
\xymatrix{
{_H\Mm} \ar@<.5ex>[rr]^-{{_H\Hom}(H,-)} && {_{H_{par}}\Mm} \ar@<.5ex>[ll]^-{H\ot_{H_{par}}-}
}.
\]
One can now easily observe that both the functors $H\ot_H-$ and ${_H\Hom}(H,-)$ are exactly the inclusion functor $i$. 
On the other hand we have a canonical monomorphism 
$$\alpha:{_{H_{par}}\Hom}(H,M)\to M, f\mapsto f(1),$$
and a canonical epimorphism
$$\beta: M\to H\otimes_{H_{par}} M, m\mapsto 1\ot_{H_{par}} m.$$
One easily checks that the image of $\alpha$ is exactly the global core of $M$ and the image of $\beta$ is the global shadow.
\end{proof}

Clearly, a single global $H$-module can be the core or shadow of many partial $H$-modules. 
Hence, it follows from the previous proposition that the inclusion functor and its adjoints do not allow to describe all partial modules. In the next section, we will construct other functors between partial and global modules that will allow such a description.

We finish this section by an explicit description of the partial modules in two illustrative examples.

\begin{exmp} \label{exampledualkC2} \exlabel{dualC2}
 Let $k$ be a field with characteristic not  equal to 2, consider the cyclic group $\Z_2$ and consider the group algebra $k \Z_2$ with base $\{u_0,u_1\}$. Recall that $\phi:k\Z_2\to (k\Z_2)^*=H$ given by $\phi(u_0)=p_0+p_1$ and $\phi(u_1)=p_0-p_1$ is a Hopf algebra isomorphism, where $\{p_0,p_1\}$ is the dual base for $\{u_0,u_1\}$. 
Then it was shown in \cite{ABV} that $H_{par}$ is isomorphic to the $3$-dimensional algebra 
\[
k[x]/\langle x(2x-1)(x-1) \rangle ,
\] 
where $x$ corresponds to the element $[p_0]$ in $H_{par}$.

A partial module over $H$, can therefore be identified with a pair 
$(\mathbb{V}, t)$, where $\mathbb{V}$ is a $k$-vector space with 
a fixed linear transformation $t:\mathbb{V}\rightarrow \mathbb{V}$ satisfying the constraint $t\circ (t-I)\circ (2t-I) =0$. 
Indeed, $\VV$ becomes this way an $H_{par}$-module by the action
\[
(a_0 +a_1 x +a_2 x^2 )\triangleright v =a_0 v + a_1 t(v) + a_2 t^2 (v).
\]
The above constraint on $t$ leads to the conclusion that the only possible eigenvalues of $t$ are $0$, $1$ and $\frac{1}{2}$, and the vector space $\mathbb{V}$ is decomposed into a threefold direct sum
\[
\mathbb{V}=\mathbb{V}_0 \oplus \mathbb{V}_1 \oplus \mathbb{V}_{\frac{1}{2}} .
\]
Hence the partial representation $\pi:H\to \End(\VV)$ may be presented by  
the block matrices (where one numbers rows and columns by $i,j=0,1,1/2$ and the $(i,j)$-block refers to $\Hom(\VV_i,\VV_j)$)
\[
\pi(p_0) =\left( \begin{array}{ccc} 1 & 0 & 0 \\ 0 & 0 & 0 \\ 0 & 0 & \frac{1}{2} \end{array} \right) ,
\quad
\pi(p_1) =\left( \begin{array}{ccc} 0 & 0 & 0 \\ 0 & 1 & 0 \\ 0 & 0 & \frac{1}{2} \end{array} \right).
\]
Since modules over $(k\Z_2)^*$ coincide with $\Z_2$-graded modules, we call a partial $H$-module a ``partially $\Z_2$-graded vector space''. 
Notice that $\VV_0\oplus \VV_1$ is exactly the global core and shadow of $\VV$.
\end{exmp}

\begin{exmp} \label{examplesweedler} \label{exampleinfinitesweedler} \exlabel{exampleSweedler}
Consider the Sweedler Hopf algebra $H_4$, which is a $4$-dimensional algebra with base $\{1,g,x,y\}$, where
\[
\begin{array}{lll}
g^2=1, & x^2=0, & gx=y=-xg,\\
\Delta (g)=g\otimes g, & \Delta (x)=g\otimes x +x\otimes 1, & \Delta(y)=1\ot y+y\ot g,\\
\epsilon (g) =1,& \epsilon (x)=0, & \epsilon(y)=0,\\ 
S(g) =g, & S(x)=-y & S(y)=x.
\end{array}
\] 
Suppose that we work over an algebraically closed base field $k$ of $\mathrm{ch}(k) \neq 2$. 
A (global) representation of $H_4$ on an $n$-dimensional vector space $V$, is then completely determined by two $n\times n$ matrices $[g]$ and $[x]$, satisfying equations as above. In particular, since $g^2=1$, it is always possible to find a base for $V$ such that $g$ takes the form of a block matrix (where the dimensions of the blocks sum up to $n$)
\[
\left(
\begin{array}{cc}
1 & 0\\ 0 & -1
\end{array}
\right) .
\]
And as a consequence of the above equations, $x$ needs to take the form of a block matrix (with blocks of the same size)
\[
\left(
\begin{array}{cc}
0 & a\\ b & 0
\end{array}
\right) , 
\]
where $ab=ba=0$. 

On the other hand, a partial representation on the same $n$-dimensional vector space $V$ would be determined by {\em three} matrices $[g]$, $[x]$, $[y]$, since the identity $[y]=[g][x]$ is not longer guaranteed in the partial case. Considering the axiom (PR2) in the case $h=k=g$ one obtains that the matrix $[g]$ should satisfy
$$[g]^3=[g].$$
Hence, we can always find a base of $V$ such that $[g]$ is a block matrix of the form
\[
\left(
\begin{array}{ccc}
1 & 0&0\\ 0 & -1&0 \\ 0&0&0
\end{array}
\right) .
\]
By comparing with the global case, the partiality of the representation is encoded by the eigenvalue $0$. 
We want to show that moreover the partiality is also only visible in the eigenspace associated to $0$. 

Let us make this more precise. In the next table some of the axioms of partial representations are computed explicitly in terms of the matrices $[x],[y],[g]$
\begin{center}
\begin{tabular}{|r|r|l|l|}
\hline
(PR3) & $h=x$, $k=g$ & $[g][y][g]=-[g][x]$ & (1)\\
(PR2) & $h=g$, $k=y$ & $[g][y][g]=[x][g]$ & (2)\\
(PR4) & $h=g$, $k=y$ & $[g][x][g]=[y][g]$ & (3)\\
(PR5) & $h=x$, $k=g$ & $[g][x][g]=-[g][y]$ & (4)\\
(PR5) & $h=k=x$ & $[g][x]^2=0$ & (5)\\
(PR4) & $h=k=x$ & $[x][g][x]=[x][y]-[y][x]$ & (6)\\
(PR2) & $h=k=x$ & $[x][g][y]=[x]^2-[y]^2$ & (7)\\
\hline
\end{tabular}
\end{center}
Then from the equations (1)-(2)-(3)-(4) it follows easily that both $[x]$ and $[y]$ anti-commute with $[g]$, and hence they are block matrices of the form
$$[x]=
\left(
\begin{array}{ccc}
0 & a&0\\ b & 0&0 \\ 0&0&c
\end{array}
\right) ,  
\qquad
[y]=
\left(
\begin{array}{ccc}
0 & a'&0\\ b' & 0&0 \\ 0&0&d
\end{array}
\right) .
$$
By equation (1), it then follows that 
$$a=a', \quad b=-b'$$
and equation (5) moreover implies that 
$$ab=ba=0.$$
Finally, by (6) and (7) and  we find that
$$cd=dc; \quad c^2=d^2.$$
Hence, any partial $H_4$-module can be decomposed in a direct sum
$$V=U\oplus W,$$
where $U$ is the global core (and shadow) of $V$ consisting of the direct sum of the eigenspaces of $[g]$ with eigenvalues $-1$ and $1$; 
and $W$ is a purely partial $H_4$-module coinciding with the $0$-eigenspace of $[g]$.

Let us give a particular example of a (purely) partial $H_4$-module. Consider an $n$-dimensional $k$-vector space $\mathbb{W}_n$ with a base $\{w_1 , \ldots w_n \}$ and define $\pi_n :H_4 \rightarrow \mbox{End}_k (\mathbb{W}_n )$ as
\[
\begin{array}{lll}
\pi_n (1) = \mbox{Id}_{\mathbb{V}_n}, &  
\pi_n (g) (w_k ) =0, \mbox{  for  } 1\leq k \leq n \\ 
\pi_n (x) (w_k ) =\left\{ \begin{array}{cc} w_{k+1} & \mbox{ for } k<n \\ 0 & \mbox{ for } k=n \end{array}\right.  
& \pi_n (gx) (w_k ) =\left\{ \begin{array}{cc} w_{k+1} & \mbox{ for } k<n \\ 0 & \mbox{ for } k=n \end{array}\right. .
\end{array}
\]
This means in the notation introduced above that 
$$c=c'=\left( \begin{array}{ccccc} 
0 & 0 & 0& \cdots & 0\\
1 & 0& 0&  \cdots & 0\\
0&1& 0& \ddots \\
&& \ddots &\ddots \\
0& \cdots && 1 & 0\end{array} \right) .$$
Furthermore, the submodules of $\WW_n$ are exactly of the form $\WW_m$ with $m<n$. Indeed, if $w=\sum_{i=1}^n a_iw_i$ is in a submodule $\WW$ of $\WW_n$ such that $a_1=\ldots=a_{k-1}=0$ and $a_{k}\neq 0$, then we find for any $\ell<n-k$ that $x^\ell w=\sum_{i=k+\ell}^na_iw_i\in \WW$ and hence $\WW$ contains $\WW_{n=k+1}$. Therefore $\WW_n$ is a tower 
$$\WW_1\subset \WW_2\subset \cdots \subset \WW_{n-1}\subset \WW_n,$$
and $\WW_n$ does not contain other submodules than these. 
It now follows that for $n>1$, $\WW_n$ is not semi-simple, hence $(H_4)_{par}$ is not semi-simple.
\end{exmp}

Clearly, if the universal Hopf algebroid $H_{par}$ is finite dimensional and semi-simple, then also $H$ is finite dimensional (since $H$ is a direct summand of $H_{par}$) and semi-simple (since any global module is a partial module and all submodules of a global module are again global). 

Conversely, note that even though $H_4$ is finite dimensional, its universal partial ``Hopf'' algebra is  infinite dimensional. Intuitively, the ``potential'' of a Hopf 
algebra to produce partial representations is measured by the size of the subalgebra  $A_{par}\subseteq H_{par}$.
In fact, it is expected that partial representations depend not only on the algebra $H$, but also on its irreducible representations. This is supported by the explicit description of the partial group algebra in \cite{DEP}.

Therefore we make the following conjecture, which we hope to prove (or disprove) in a following paper.

\begin{conjecture}
If $H$ is a finite dimensional semi-simple Hopf algebra, then $H_{par}$ is also semi-simple and finite dimensional.
\end{conjecture}

\section{Projections and partial modules}

It is known that partial actions can be constructed by restricting global actions to suitable subalgebras. As a first result, we will now show that partial modules of a Hopf algebra $H$ can be constructed from ordinary $H$-modules by means of projections.

\begin{defi} Let $H$ be a Hopf algebra, $M$ a left $H$-module with the module structure given by the 
map $\triangleright :H\otimes M \rightarrow M$ and $T:M \rightarrow M$ a $k$-linear projection (i.e.\ $T^2=T$). We say that $T$ satisfies the {\em
$c$-condition} if 
for all $h\in H$ we have
$$T_h\circ T=T\circ T_h ,$$
in which $T_h =h_{(1)}\triangleright T(S(h_{(2)})\triangleright -)$ for $h\in H$.

Similarly, we say $T$ satisfies the $\tilde c$-condition if 
for all $h\in H$ we have
$$\tilde T_h\circ T=T\circ \tilde T_h , $$
in which $\tilde{T}_h =S(h_{(1)}) \triangleright T(h_{(2)}\triangleright -)$ for $h\in H$.
\end{defi}

Remark that $T_h$ is nothing else than $h\blacktriangleright T$, where $\blacktriangleright$ denotes the left adjoint action of $H$ on $\Hom_k(H,M)$.

\begin{lemma}
For a left $H$-module and a projection $T\in\End_k(M)$, the following statements are equivalent:
\begin{enumerate}[(i)]
\item $T$ satisfies the $c$-condition;
\item $T$ satisfies the $\tilde c$-condition;
\item $T_h\circ \tilde T_k = \tilde T_k\circ T_h$ for all $h,k\in H$.
\end{enumerate}
\end{lemma}

\begin{proof}
\ul{$(i)\Rightarrow(iii)$}. Consider $m\in M$, then
\begin{eqnarray*}
T_h \circ \tilde{T}_k (m) & = & h_{(1)} \triangleright T(S(h_{(2)})S(k_{(1)})\triangleright T(k_{(2)} \triangleright m)) \\
& = & h_{(1)}S(k_{(1)}h_{(2)}) k_{(2)}h_{(3)} \triangleright T(S(k_{(3)}h_{(4)})\triangleright T (k_{(4)}\triangleright m)) \\
& = & h_{(1)}S(k_{(1)}h_{(2)}) \triangleright (k_{(2)}h_{(3)} \triangleright T(S(k_{(3)}h_{(4)})\triangleright T (k_{(4)}\triangleright m))) \\
& = & S(k_{(1)}) \triangleright (T_{k_{(2)}h} \circ T (k_{(3)}\triangleright m))\\
& = & S(k_{(1)}) \triangleright (T \circ T_{k_{(2)}h}(k_{(3)}\triangleright m))\\
& = & S(k_{(1)}) \triangleright T( k_{(2)}h_{(1)} \triangleright T(S(k_{(3)}h_{(2)}) k_{(4)})\triangleright m))) \\
& = & S(k_{(1)}) \triangleright T( k_{(2)}\triangleright (h_{(1)}\triangleright T(S(h_{(2)})\triangleright m)))\\
& = & \tilde{T}_k \circ T_h (m).
\end{eqnarray*}
\ul{$(iii)\Rightarrow (ii)$}. Follows immediately by considering $T=T_1$.\\
\ul{$(ii)\Rightarrow (iii)\Rightarrow (i)$}. Follows by applying the previous implications on the Hopf algebra $H^{cop}$ which has the same underlying algebra as $H$, but the opposite comultiplication (i.e.\ $\Delta_{H^{cop}}(h)=h_{(2)}\ot h_{(1)}$). Observe that the antipode of $H^{cop}$ is given by the inverse of the antipode of $H$ and that the transition from $H$ into $H^{cop}$ reverses the roles of $T_h$ and $\tilde T_h$.
\end{proof}

\begin{prop} \label{projections}\prlabel{projections} Let $H$ be a Hopf algebra over $k$, $M$ a left $H$-module and $T\in \End_k (M)$  a $k$-linear projection satisfying the $c$-condition.
Then the linear map $\pi :H\rightarrow \End_k (T(M))$ given by 
$\pi (h) (m)=T(h\triangleright m)$, in which $m=T(m')$, for some $m'\in M$, defines a partial representation of the Hopf algebra $H$, i.e. it endows $T(M)$ with the structure of a partial $H$-module.
\end{prop}

\begin{proof} In order to verify whether $\pi$ is a partial representation, one needs to check axioms (PR1) up to (PR5). Take $m\in T(M)$, then $T(m)=m$ and (PR1) reads
\[
\pi (1_H) (m) = T(1_H \triangleright m)=T(m)=m .
\]
For (PR2), we find
\begin{eqnarray*}
\pi (h) \pi (k_{(1)}) \pi (S(k_{(2)})) (m)  & = & T(h\triangleright T(k_{(1)} \triangleright
T( S(k_{(2)}) \triangleright m))) \\
& = & T(h \triangleright T\circ T_{k} (m)) =  T(h\triangleright T_{k} \circ T(m)) \\
& = & T(h\triangleright T_{k} (m)) \\
& = & T(hk_{(1)} \triangleright T(S(k_{(2)})\triangleright m))\\
& = & \pi (hk_{(1)}) \pi (S(k_{(2)})) (m) .
\end{eqnarray*}
The other items are proved in a similar way. Therefore, $\pi :H\rightarrow \End_k (T(M))$ is a partial representation, which makes $(T(M), \pi )$ into 
a partial $H$-module
\end{proof}

\begin{exmp}\exlabel{trivial}
Let $M$ be an $H$-module. Then clearly $id_M$ satisfies the $c$-condition and hence we recover the trivial observation that any global $H$-module is also a partial $H$-module. 

Furthermore, if some $T\in\End_k(M)$ satisfies the $c$-condition, then also $T'=id_M-T$ satisfies the $c$-condition. Indeed, for any $h\in H$ we have
\[
T'_h (m)=h_{(1)}\triangleright T'(S(h_{(2)})\triangleright m)=h_{(1)}S(h_{(2)})\triangleright m -T_h (m)=\epsilon (h)m-T_h (m),
\]
which clearly satisfies again the $c$-condition. Therefore, each projection on an $H$-module $M$ satisfying the $c$-condition splits $M$ into two direct summands, each one being a partial $H$-module.
\end{exmp}

\begin{exmp} Let $H=kG$, the group algebra of the group $G$. A $kG$-module is given by a classical representation $u: G\rightarrow GL(\mathbb{V})$. 
Given a projection $T:\mathbb{V}\rightarrow \mathbb{V}$ which commutes with 
$T_g =u(g)Tu(g^{-1})$, for all $g\in G$, the mapping $v: G \rightarrow \End_k (\mathbb{V})$, defined as $v(g)=Tu(g)|_{T(\mathbb{V})}$ is a 
partial representation of the group $G$, as proved in Proposition 1.2 of \cite{Abadie2}.
\end{exmp}

\begin{exmp} Let $B$ be an $H$-module algebra and $e\in B$ be a central idempotent. Consider the $k$-linear projection $T:B\to B,\ T(a)=ea=ae$ for all $a\in B$. Then we find that
\[
T_h (a)= h_{(1)} \triangleright (e(S(h_{(2)})\triangleright a)) .
\]
For $h\in H$ and $a\in B$ we have
\begin{eqnarray*}
T_h (ea) & = & h_{(1)} \triangleright (e(S(h_{(2)})\triangleright ea)) =(h_{(1)}\triangleright e)(h_{(2)}S(h_{(3)})\triangleright (ea)) \\
& = & (h\triangleright e)ea =e(h\triangleright e)a \\
& = & e(h_{(1)}\triangleright e)(h_{(2)}S(h_{(3)})\triangleright a) =e(h_{(1)} \triangleright (e(S(h_{(2)})\triangleright a)))\\
& = & e(T_h (a)) .
\end{eqnarray*}
Therefore, $T$ satisfies the $c$-condition and by \prref{projections}, we obtain that $A=eB$  is a partial $H$-module by
$$\pi (h) (a)=e(h\triangleright a)$$ 
for every $a\in A$. In fact, $A$ is also an algebra (with unit $1_A=e$) and one can observe that with these structures, $A$ is a partial $H$-module algebra, which was proven in \cite{AB}.  
\end{exmp}

\begin{exmp} \label{Z2grading}\exlabel{dualC2_2} Consider the Hopf algebra $H=(k\mathbb{Z}_2)\cong (k\mathbb{Z}_2)^*$. An $H$-module $\VV$ is just a $\mathbb{Z}_2$-graded $k$-vector space. Let us write
\[
\mathbb{V} =\mathbb{V}_0 \oplus \mathbb{V}_1 = k^{n_1 +t} \oplus k^{n_2 +t} .
\]
Using the notation from \exref{dualC2}, the action of $(k\mathbb{Z}_2)^*$ on $\VV$ is given by assigning $p_0$ and $p_1$ to the projections over the subspaces $\mathbb{V}_0$ and $\mathbb{V}_1$, respectively. 
For the subspace $\mathbb{V}_0$, consider the basis $\{ u_1 ,\ldots , u_{n_1} \} \cup \{ e_1 ,\ldots e_t \}$, and for the space $\mathbb{V}_1$, the 
basis $\{ v_1 ,\ldots , v_{n_2} \} \cup \{ f_1 ,\ldots f_t \}$. Take the $k$-linear transformation $T:\mathbb{V}\rightarrow \mathbb{V}$ given by
\[
T(u_i) =u_i ; \quad T(v_j )=v_j; \quad T(e_k )=T(f_k )=\frac{1}{2}(e_k +f_k ) ,
\]
for $1\leq i\leq n_1$, $1\leq j \leq n_2$ and $1\leq k \leq t$. It is easy to see that $T$ is a projection. Now, consider the $k$-linear operators 
\[
T_0 =p_0 \circ T \circ p_0 +p_1 \circ T \circ p_1 ,
\]
and
\[
T_1 =p_0 \circ T \circ p_1 + p_1 \circ T \circ p_0 ,
\]
which, in the basis of $\mathbb{V}$ read
\[
\begin{array}{cccc} T_0 (u_i )=u_i; & T_0 (v_j ) =v_j ; & T_0 (e_k) =\frac{1}{2} e_k ;& T_0 (f_k )= \frac{1}{2} f_k ,\\
T_1 (u_i )=0; & T_1 (v_j ) =0 ; & T_1 (e_k) =\frac{1}{2} f_k ;& T_1 (f_k )= \frac{1}{2} e_k .
\end{array}
\]
A straighforward calculation shows us that $T$ commutes with $T_0$ and $T_1$. 
By Proposition \ref{projections}, we therefore obtain that if $\mathbb{W}=T(V)$ then $\pi :(k\mathbb{Z}_2)^* \rightarrow \End_k (\mathbb{W})$ is a partial representation of the Hopf algebra $(k\mathbb{Z}_2)^*$ which is given by
$\pi (p_0 )=T\circ p_0$ and $\pi (p_1 )=T\circ p_1$. More precisely, writing 
the basis of $\mathbb{W}$ as $\{u_i \}_{1\leq i\leq n_1} \cup 
\{ v_j \}_{1\leq j \leq n_2}\cup \{ g_k= e_k + f_k \}_{1\leq k\leq t}$, we have
\[
\begin{array}{ccc} \pi (p_0 ) (u_i )=u_i; & \pi (p_0 ) (v_j ) =0 ; & \pi (p_0 ) (g_k) =\frac{1}{2} g_k ,\\
\pi (p_1 )(u_i )=0; & \pi (p_1) (v_j ) =v_j ; & \pi (p_1 ) (g_k) =\frac{1}{2} g_k .
\end{array}
\]
One can easily observe that $\WW$ with the linear operator $t=\pi(p_0)$ satisfies the conditions of \exref{dualC2}.
\end{exmp}

\begin{exmp} \label{sweedler} \exlabel{Sweedler2} Consider the four dimensional Sweedler Hopf algebra 
$H_4$. As we already observed in \exref{exampleSweedler}, a finite dimensional $H_4$-module $M$ is determined by matrices $[g],[x]\in \End_k(M)$. 
Suppose that $\dim M=2n$ and that the eigenspaces for $[g]$ with eigenvalue $1$ and $-1$ are both of dimension $n$. Then we can choose a base of $M$ such that $[g]$ has the form
\[
[g]=\left( \begin{array}{cc} 0 & I_n \\ I_n & 0 \end{array} \right) .
\]
Since $[g][x]=-[x][g]$ we then find that $[x]$ can be written as a block matrix
\[ 
[x]=\left( \begin{array}{cc} c & -d \\ 
d & -c \end{array} \right) ,
\]
and since $[x]^2=0$ we have $cd=dc$ and $c^2=d^2$. From the characterizations of partial $H_4$-modules in \exref{exampleSweedler}, it follows that an $n$-dimensional space can be endowed with the structure of a partial $H_4$-module by means of  
$$[g]=0, \quad [x]=c, \quad [gx]=d.$$
We aim to show that this partial module arises exactly as the restriction of the global $H_4$-module $M$ by means of the projection
$$T=\left( \begin{array}{cc} I_n & 0 \\ 0 & 0 \end{array} \right),$$
i.e. the projection on the subspace $W$ of $M$ generated by the first $n$ base vectors.
Indeed, consider the operators
\begin{eqnarray*}
T_g  & = & gTg =\left( \begin{array}{cc} 0 & 0 \\ 0 & I_n \end{array} \right) , \\
T_x  & = & xT -gTgx =\left( \begin{array}{cc} c & 0 \\ 0 & c \end{array} \right) , \\
T_{gx}  & = & gxTg +Tx =\left( \begin{array}{cc} c & 0 \\ 0 & c \end{array} \right) . 
\end{eqnarray*}
Obviously $T$ commutes with $T_g$, $T_x$ and $T_{gx}$, i.e.\ $T$ satisfies the $c$-condition. Therefore, by Proposition \ref{projections} we have that 
$\pi :H_4 \rightarrow \End_k (W)$ given by
\begin{eqnarray*}
\pi (g) (v) & = & Tg \left( \begin{array}{c} v \\ 0 \end{array} \right) =0 ,\\
\pi (x) (v) & = & Tx \left( \begin{array}{c} v \\ 0 \end{array} \right) 
= cv , \\
\pi (gx) (v) & = & Tgx \left( \begin{array}{c} v \\ 0 \end{array} \right) 
= dv ,
\end{eqnarray*}
defines a $n$-dimensional partial representation of the Sweedler's Hopf algebra $H_4$, which is exactly a generic purely partial $H_4$-module as described in \exref{exampleSweedler}.
\end{exmp} 

\begin{defi} Let $H$ be a Hopf $k$-algebra, the category of $H$-modules with projections is the category $\mathbb{T}({}_H \mathcal{M})$ whose objects 
consist of pairs $(M,T)$, in which $M$ is a left $H$-module and $T$ is a $k$-linear projection in $M$ satisfying the $c$-condition. A morphism $f: (M,T)\rightarrow (N,S)$ in this 
category, is a $k$-linear map $f:TM\rightarrow SN$ satisfying 
$$f(T(h\triangleright m)) = S(h\triangleright f(m))$$
for all $m\in TM$. 
\end{defi}

Let $M$ be a left $H$-module and $N\subset M$ any $k$-linear subspace. Then the $H$-submodule generated by $N$ is the set 
$$H\triangleright N=\{\sum_i h_i\triangleright n_i~|~h_i\in H, n_i\in N\},$$
which is the smallest $H$-submodule of $M$ containing $N$.

\begin{lemma}\lelabel{minimal}
Let $(M,T)\in \TT(_H\Mm)$ be an $H$-module with projection. Then $(M,T)\cong (H\triangleright TM,T)$ in $\TT(_H\Mm)$. Consequently, if $M'\subset M$ is an $H$-submodule such that $TM'=0$, then $(M,T)\cong (M/M',T)$ in $\TT(_H\Mm)$.  
\end{lemma}

\begin{proof}
First observe that $TM=T(H\triangleright TM)$. Indeed, since $H\triangleright TM\subset M$ clearly $T(H\triangleright TM)\subset TM$. On the other hand, for any $T(m)\in TM$ we find that 
$$T(m)=T^2(m)=T(1_H\triangleright T(m))\in T(H\triangleright TM).$$
Then clearly $id_{TM}$ is an isomorphism between $(M,T)$ and $(H\triangleright TM,T)$ in $\TT(_H\Mm)$.

For the second statement, remark that since $TM'=0\subset M'$, $T$ induces a well defined projection $M/M'\to M/M'$ (which we will denote again by $T$) and the pair $(M/M',T)$ is again an element of $\TT(_H\Mm)$. Furthermore, since $H\triangleright TM=H\triangleright T(M/M')$, it follows from the first part of the Lemma that the canonical projection $M\to M/M'$ defines an isomorphism $(M,T)\cong (M/M',T)$ in $\TT(_H\Mm)$.
\end{proof}

As a consequence of the previous lemma, we can always consider an object $(M,T)$ in $\TT(_H\Mm)$ to be `minimal' in the the sense that $M$ does not contain any $H$-submodules that are annihilated by $T$. 

\begin{proposition}\prlabel{Rff}
With notations as above, the inclusion functor $i:{}_H\Mm\to {}_H\Mm^{par}$ from \exref{trivialmod} factors as a composition of fully faithful functors
$$\xymatrix{
{}_H\Mm \ar[rr]^-{I} && \TT(_H\Mm)  \ar[rr]^-{R} && {}_H\Mm^{par}
}$$
defined by $I(M)=(M,id_M)$ and $R(M,T)=(TM,\pi_{TM})$ where $\pi_{TM}:H\to \End(TM),\ \pi_{TM}(h)(m)=T(h\triangleright m)$ for all $h\in H$ and $m\in TM$. 

The functor $R$ will be called the {\em restriction functor}.
\end{proposition}

\begin{proof}
Obviously, $I$ is a fully faithful functor (see also \exref{trivial}). We know from \prref{projections} that $R(M,T)$ is indeed a partial $H$-module, and by definition a morphism $f:(M,T)\to (N,S)$ in $\TT(_H\Mm)$ is exactly a morphism $f:TM\to TS$ in ${}_H\Mm^{par}$.
\end{proof}

The problem of globalization, which will be discussed in the next sections 
consists in constructing a left adjoint to the inclusion functor $i$, which moreover corestricts to an equivalence between ${_H\Mm^{par}}$ and $\TT(_H\Mm)$.

\section{Dilations}

\begin{defi} Let $H$ be a Hopf algebra. A {\em dilation} of a partial $H$-module $(M,\pi)$ is a pair $((N, T), \theta )$ where $(N,T)$ is an object of the category $\mathbb{T}({}_H \mathcal{M})$ and $\theta:(M,\pi)\to R(N,T)$ is an isomorphism of partial $H$-modules.

We say that a dilation $((N,T), \theta )$ is {\em proper} if $N$ is generated by $TN=\theta(M)$ as an $H$-module and we say it is {\em minimal} if $N$ does not contain any $H$-submodule that is annihilated by $T$.
\end{defi}

\begin{remarks}
Notice that $\theta:M\to TN$ will be a morphism of partial $H$-modules if and only if  $\theta (\pi (h)(m))=T(h\triangleright \theta (m))$, for every $m\in M$ and $h\in H$. 
Furthermore, another way to understand a dilation is as a quadruple $(\theta:M \leftrightarrows N:\varpi)$ where $M$ is a partial $H$-module, $N$ is a global $H$-module, and $\theta$ and $\varpi$ are linear maps satisfying $\varpi\circ \theta=id_M$ and $T=\theta\circ \varpi$ satisfies the $c$-condition and $h\cdot m=\varpi(h\triangleright \theta(m))$ for all $m\in M$ and $h\in H$.

Thanks to \leref{minimal}, we see that for any dilation $((N,T), \theta )$ the object $(N,T)$ is isomorphic to an object $(N',T')$ in the category $\mathbb{T}({}_H \mathcal{M})$, such that $((N',T'), \theta )$ is a proper and minimal dilation.
\end{remarks}

We will now present our first main result, which generalizes the globalization of partial actions to dilations of partial $H$-modules.

\begin{theorem} \label{dilationtheorem}\thlabel{dilation} Let $H$ be a Hopf algebra, then every partial $H$-module $(M,\pi)$ admits a (proper and minimal) dilation $((\ol M,T_\pi), \varphi )$. We call this the {\em standard dilation} of $M$.
\end{theorem}

\begin{proof} Let $(M,\pi )$ be a partial $H$-module and consider the $k$-vector space $\Hom_k (H,M)$. This vector space has a natural $H$-module 
structure $\triangleright :H\otimes \Hom_k (H,M)\rightarrow \Hom_k (H,M)$ given by
\[
(h\triangleright f)(k)=f(kh).
\]
Define the linear map $\varphi :M\rightarrow \Hom_k (H,M)$ by
\[
\varphi (m)(h)=\pi (h)(m) .
\]
It is easy to see that this map is injective. Indeed, take $m\in \mbox{Ker} (\varphi)$, i.e.\ $\varphi (m)(h)=0$ for all $h\in H$. Then
\[
m=\pi (1_H )(m)=\varphi (m)(1_H) =0 .
\]
Denote by $\overline{M}$ the $H$-submodule of $\Hom_k (H,M)$ generated by the image of $\varphi$, that is $\overline{M}=H\triangleright \varphi (M)$. Writing $f=\sum_i h_i \triangleright \varphi (m_i )$, a typical element of $\ol M$, we obtain for all $k\in H$
\begin{eqnarray} 
f(k)&=&\sum_i\varphi(m_i)(kh_i)=\sum_i\pi(kh_i)(m_i) . \eqlabel{eq1}
\end{eqnarray}
Applying this twice we find furthermore that
\begin{eqnarray}
\pi(k_{(1)})f(S(k_{(2)}))&\stackrel{\equref{eq1}}{=}& \sum_i \pi(k_{(1)})\pi(S(k_{(2)})h_i)(m_i) \stackrel{(PR3)}{=} 
 \pi(k_{(1)})\pi(S(k_{(2)}))\pi(h_i)(m_i) \nonumber \\
 &\stackrel{\equref{eq1}}{=}& \pi(k_{(1)})\pi(S(k_{(2)}))(f(1_H)) . \eqlabel{eq2}
\end{eqnarray} 
We now define the linear map 
$$T_{\pi}: \overline{M}\rightarrow \overline{M},\ T_{\pi}(f)=\varphi (f(1_H))
$$
where, as before, $f=\sum_i h_i \triangleright \varphi (m_i )\in\overline{M}$. 
Let us verify that $T_{\pi}$ is a projection:
\[
T_{\pi} \circ T_{\pi} (f)=\varphi (T_{\pi}(f)(1_H))=\varphi (\varphi (f(1_H))(1_H))=\varphi (\pi (1_H)( f(1_H))) =\varphi (f(1_H))=T_{\pi}(f) .
\]
By definition, it is obvious that $T_{\pi}(\overline{M})\subseteq \varphi (M)$, on the other hand, if $m\in M$ then
\[
\varphi (m)=\varphi (\pi (1_H)(m))=\varphi (\varphi (m)(1_H))=T_{\pi} (\varphi (m)) .
\]
Therefore $\varphi (M)=T_{\pi}(\overline{M})$.

To check that $T_{\pi}$ satisfies the $c$-condition, consider for 
each $h\in H$ the linear operator $T_h :\overline{M} \rightarrow \overline{M}$ which is given by
\[
T_h (f)= h_{(1)}\triangleright T_{\pi} (S(h_{(2)})\triangleright f) .
\]
Then  we find 
\begin{eqnarray}
T_h (f)(k)&=&  \big(h_{(1)}\triangleright T_{\pi} (S(h_{(2)})\triangleright f)\big)(k)
= \big(T_{\pi} (S(h_{(2)})\triangleright f)\big)(kh_{(1)}) \nonumber \\
&=& \varphi((S(h_{(2)})\triangleright f)(1_H))(kh_{(1)}) 
= \varphi(f(S(h_{(2)})))(kh_{(1)}) \nonumber\\
&=& \pi(kh_{(1)})(f(S(h_{(2)}))) . \eqlabel{eq3}
\end{eqnarray}

Take $h,k\in H$ and $f\in \overline{M}$, then
\begin{eqnarray*}
T_{\pi} \circ T_h (f)(k) & = & \varphi (T_h (f)(1_H))(k) = \pi(k)(T_h (f)(1_H)) \\
&\stackrel{\equref{eq3}}{=}& 
\pi(k)\pi(h_{(1)})(f(S(h_{(2)}))) \stackrel{\equref{eq2}}{=}
\pi(k)\pi(h_{(1)})\pi(S(h_{(2)}))(f(1_H)) \\
&\stackrel{(PR2)}{=}& \pi(kh_{(1)})\pi(S(h_{(2)}))(f(1_H))
= \pi (kh_{(1)}) (\varphi (f(1_H))(S(h_{(2)})))\\ &=& \pi (kh_{(1)}) \big(T_{\pi}(f)(S(h_{(2)}))\big)
\stackrel{\equref{eq3}}{=} T_h \circ T_{\pi} (f)(k) .
\end{eqnarray*}

By Proposition \ref{projections}  $T_{\pi}(\overline{M})=\varphi (M)$ is a partial $H$-module via
$\pi_{T_{\pi}} :H\rightarrow \End_k (\overline{M})$, given by 
\[
\pi_{T_{\pi}}(h) =T_{\pi}(h\triangleright -) .
\]

Moreover
\[
\pi_{T_{\pi}}(h)\varphi (m)=T_{\pi}(h\triangleright \varphi (m))=\varphi ((h\triangleright \varphi (m))(1_H))=\varphi (\varphi (m)(h))=
\varphi (\pi (h)(m)), 
\]
i.e. \ $\varphi$ is a morphism of partial $H$-modules. Therefore, the pair $((\overline{M} ,T_{\pi}),\varphi )$ is a dilation. Let us verify that this dilation is minimal. 
Take any $f=\sum_i h_i \triangleright \varphi (m_i ) \in \ol M$ such that $T_{\pi}(h\triangleright f)=0$ for every $h\in H$. Then we find 
\begin{eqnarray*}
0&=& T_{\pi}(h\triangleright f) = \varphi((h\triangleright f)(1_H)) = \varphi(f(h)) .
\end{eqnarray*}
As the map $\varphi$ is injective, we conclude that $f(h)=0$ for every $h\in H$, then $f=0$. Therefore 
the dilation 
$((\overline{M} ,T_{\pi}),\varphi )$ is minimal.
\end{proof}

We now can easily prove the main result of this paper.

\begin{theorem}\thlabel{Requivalence}
The restriction functor $R:\TT(_H\Mm)\to {_H\Mm^{par}}$ is an equivalence of categories.
\end{theorem}

\begin{proof}
We already know by \prref{Rff} that $R$ is fully faithful. Furthermore, by \thref{dilation} it follows that $R$ is essentially surjective and we can conclude that $R$ is an equivalence of categories.
\end{proof}

Note that the quasi-inverse of the restriction functor $R$ assigns to a partial module the $H$-module with projection from its standard dilation.

\begin{proposition}\prlabel{propdilationfunctor}
The construction of the standard dilation defines a functor
$$D:{_H\Mm^{par}}\to {_H\Mm},\quad D(M)=\ol M$$ 
that we call the {\em dilation functor}.
Moreover the functor $D$ is additive, faithful and it preserves monomorphisms, epimorphisms and arbitrary coproducts.
\end{proposition}

\begin{proof}
First, let us verify that the construction of minimal dilations is coherent with morphisms. Let $f: (M, \pi ) \rightarrow (N, \rho )$ be a morphism in the category of partial $H$-modules, i.e.\ $f\circ \pi (h) =\rho (h) \circ f$ for every $h\in H$. Let $(\overline{M}, \overline{T}), \varphi )$ and $((\overline{N}, \overline{S}), \theta )$ be, respectively, the dilations of $(M, \pi )$ and $(N, \rho )$ as constructed above. Define the linear map $\overline{f}: \overline{M} \rightarrow \overline{N}$ as
\begin{equation}\eqlabel{Df}
\overline{f} (\sum_i h_i \triangleright \varphi (m_i ))= \sum_i h_i \triangleright \theta (f(m_i )) .
\end{equation}
Let us verify that $\overline{f}$ is well defined as a linear map. Consider a linear combination $\sum_i h_i \triangleright \varphi (m_i )=0$ in $\overline{M}$, then we  conclude that, for any $h\in H$, we have
\[
\sum_i \pi (hh_i )(m_i )=0 .
\]
Now, take any $h\in H$, then
\begin{eqnarray*}
\overline{f} (\sum_i h_i \triangleright \varphi (m_i ))(h) & = & \sum_i h_i \triangleright \theta (f(m_i )) (h) \\
& = & \sum_i \theta (f(m_i )) (hh_i ) \\
& = & \sum_i \rho (hh_i )(f(m_i ) ) \\
& = & f(\sum_i \pi (hh_i) (m_i))=0 .
\end{eqnarray*}
Therefore, $\overline{f}$ is well defined and, by construction, $\overline{f}$ is a morphism of $H$-modules. 
Hence $D:{_H\Mm^{par}}\to {_H\Mm}$, where $D(M)=\ol M$ and $D(f)=\ol f$, is a functor.

By construction it is clear that $\ol{f+g}=\ol f+\ol g$, so the functor is additive.

To see that the dilation functor is faithful, consider the diagram
\[
\xymatrix{
M\ar[r]^f \ar@{^(->}[d]_{\varphi_M} & N \ar@{^(->}[d]_{\varphi_N}\\
\ol M \ar[r]^-{\ol f} & \ol N
}
\]
Then clearly $\ol f$ is zero if and only if $f$ is zero.

Let $f: M\to N$ be an injective morphism of partial $H$-modules and consider the dilated map
$$\ol f:\ol M\to \ol N,\ \ol f(\sum_i h_i\triangleright \varphi(m_i))=\sum_i h_i\triangleright \varphi(f(m_i)).$$
Suppose that $\ol f(\sum_i h_i\triangleright \varphi(m_i))=0$, then we find for all $k\in H$ that
\begin{eqnarray*}
0&=& \sum_i kh_i\cdot f(m_i)\\
&=& f(\sum_i kh_i\cdot m_i) , 
\end{eqnarray*}
and by the injectivity of $f$ we obtain that $\sum_i kh_i\cdot m_i=0$ for all $k\in H$, so $\sum_ih_i\triangleright \varphi(m_i)=0$ and $\ol f$ is injective.

The fact that $D$ preserves epimorphisms is clear by the formula \equref{Df}.

For the last statement, consider the following commutative diagram
\[
\xymatrix{
M_i \ar[d]^-{\varphi_i} \ar[rrrr] &&&& \coprod_{i\in I} M_i \ar[d]^-\varphi  \\
\Hom(H,M_i) \ar[rr] &&
 \coprod_{i\in I}\Hom(H,M_i) \ar@{^(->}[rr] && \Hom(H,\coprod_{i\in I}M_i) 
}
\]

Hence, although $\Hom(H,\coprod_{i\in I}M_i)$ is bigger than $\coprod_{i\in I}\Hom(H,M_i)$, the image of the map $\varphi:\coprod_{i\in I}M_i\to \Hom(H,\coprod_{i\in I}M_i)$ as well as the $H$-module generated by it lies in
$\coprod_{i\in I}\Hom(H,M_i)$.
 Therefore, the standard dilation of $\coprod_{i\in I}M_i$ is a subspace of $\coprod_{i\in I}\Hom(H,M_i)$, 
  and from this it follows easily that $\ol{\coprod_{i\in I}M_i}=\coprod_{i\in I}\ol{M_i}$.
\end{proof}

\begin{remark}
From the previous proposition it is clear that the dilation functor is close to being exact. We will investigate this in the next section.
\end{remark}

The standard dilation satisfies the following universal property.

\begin{prop} \label{uniquenessdilation}\prlabel{uniquenessdilation} Let $H$ be a Hopf algebra and $(M,\pi )$ be a partial $H$-module. 
\begin{enumerate}[(i)]
\item For any other proper dilation $((N,T), \theta)$, there is a unique surjective $H$-linear map $\Phi : N \to \overline{M}$ such that $T_\pi\circ\Phi=\Phi\circ T$ and $\Phi \circ \theta =\varphi$;
\item The morphism $\Phi$ is a injective (hence bijective) if, and only if the dilation $((N,T),\theta )$ is minimal.
\end{enumerate}
\end{prop}

\begin{proof} \ul{(i)}.  
Since $N=H\triangleright \theta (M)$, any $H$-linear morphism $\Phi:N\rightarrow \overline{M}$ satisfying $\Phi \circ \theta =\varphi$ should be given by 
$$\Phi(x)= \sum_i h_i\triangleright \varphi (m_i ),$$
where $x=\sum_i h_i \triangleright \theta (m_i)\in N=H\triangleright \theta (M)$. Let us verify that this formula is well defined, i.e.\ $\Phi$ exists and is unique. To this end, consider $x=\sum_i h_i \triangleright \theta (m_i) =0$ . Then we find for all $h\in H$ (we use that $\theta$ is a partial $H$-module map in the fourth equality)
\begin{eqnarray*}
\theta(\Phi(x)(h))&=& \sum_i \theta\big((h_i\triangleright \varphi (m_i ))(h)\big)
= \sum_i \theta\big(\varphi (m_i )(hh_i)\big) = \sum_i \theta\big(\pi(hh_i)(m_i )\big)\\
&=& T(\sum_i hh_i \triangleright \theta (m_i)) = T(h\triangleright x) = 0
\end{eqnarray*} 
and since $\theta$ is injective, it follows that $\Phi(x)(h)=0$, hence $\Phi(x)=0$. 

By construction, we have found that $\Phi$ is a surjective $H$-module morphism and $\Phi \circ \theta =\varphi$. Finally, for any $x=\sum_i h_i \triangleright \theta (m_i) \in N$,
\begin{eqnarray*}
\Phi (T(x)) & = & \Phi (T(\sum_i h_i \triangleright \theta (m_i)))=\Phi (\theta (\sum_i \pi (h_i ) (m_i)))\\
& = & \varphi (\sum_i \pi (h_i ) (m_i))=T_{\pi}(\sum_i h_i \triangleright \varphi (m_i))\\
& = & T_{\pi}(\Phi (x)) .
\end{eqnarray*}
Hence $\Phi\circ T=T_\pi\circ \Phi$.\\
\ul{(ii)}. Consider an element $x=\sum_i h_i \triangleright \theta (m_i) \in N=H\triangleright \theta(M)$.
Then we have that 
\[
T(h\triangleright x)=T(\sum_i hh_i \triangleright \theta (m_i))
=\theta (\sum_i \pi (hh_i )(m_i)).
\]
Hence by the injectivity of $\theta$, we have $T(H\triangleright x)=0$ if and only if 
$$0=\sum_i \pi (hh_i )(m_i)=\sum_i \varphi(m_i)(hh_i)=\sum_i (h_i\triangleright \varphi(m_i))(h)=\Phi(x)(h),$$ 
for all $h\in H$. In other words, $x\in \ker\Phi$ if and only if $T(H\triangleright x)=0$. 
Hence the injectivity of $\Phi$ is equivalent to the minimality of $N$.
\end{proof}

As a corollary, we find now that the category of partial $H$-modules can also be described in an alternative way.
Consider a category $\DD(_H\Mm)$ whose objects are pairs $(M,T)$ where $M$ is a global $H$-module, $T$
is a projection satisfying the $c$-condition and  $((M,T),\varphi)$, with $\varphi:TM\to M$ the canonical inclusion, is a proper minimal dilation (of the partial module $TM$). 
A morphisms $f:(M,T)\to (N,S)$ is an $H$-linear map $f:M\to N$ such that $S\circ f=f\circ T$.

\begin{corollary}
The categories ${_H\Mm^{par}}$ and $\TT(_H\MM)$ are equivalent to $\DD(_H\Mm)$.
\end{corollary}

\begin{proof}
Let us check that the construction of the standard dilation yields a well-defined functor $F:{_H\Mm^{par}}\to\DD(_H\Mm)$. It already follows from  \prref{propdilationfunctor} that this functor is well-defined on objects and sends a morphism $f:(M,\pi)\to (N,\rho)$ in ${_H\Mm^{par}}$ to an $H$-linear map $\ol f:\ol M\to \ol N$, where $((\ol M,\ol T),\varphi)$ and $((\ol N,\ol S),\theta)$ are the standard dilations. It remains to verify that $\overline{f}\circ \overline{T}=\overline{S}\circ \overline{f}$. Take an element $x=\sum_i h_i \triangleright \varphi (m_i )\in \overline{M}$, then
\begin{eqnarray*}
\overline{f}(\overline{T}(x)) & = & \overline{f} (\varphi (\sum_i (h_i \triangleright \varphi (m_i ))(1_H ))) 
 =  \overline{f} (\varphi (\sum_i  \varphi (m_i )(h_i ))) \\
& = & \overline{f} (\varphi (\sum_i  \pi (h_i ) (m_i ))) 
 =  \theta (\sum_i  f(\pi (h_i ) (m_i ))) \\
& = & \theta (\sum_i  \rho (h_i ) (f(m_i ))) 
 =  \theta (\sum_i  \theta (f(m_i ))(h_i )) \\
& = & \theta (\sum_i  (h_i \triangleright \theta (f(m_i )))(1_H )) 
 =  \theta (\overline{f} (\sum_i  h_i \triangleright \varphi (m_i ))(1_H )) \\
& = & \overline{S} (\overline{f}(x)).
\end{eqnarray*} 
On the other hand, there is an obvious functor $G:\DD(_H\Mm)\to\TT(_H\Mm)$ which is the identity on objects. Finally the composition $G\circ F:{_H\Mm^{par}}\to \TT(_H\Mm)$ is clearly the quasi-inverse of the restriction functor as described after \thref{Requivalence}. Conversely, because of \prref{uniquenessdilation}, the functor $F\circ G$ is again the identity.
\end{proof}

\begin{remark}\relabel{monoidal}
It is well-known that by transfer of structure, any equivalence $F:\mathcal C\rightleftarrows \mathcal D:G$ between a monoidal category $\Cc$ and any category $\mathcal D$ induces a monoidal structure on $\mathcal D$ turning $F,G$ into a monoidal equivalence. Hence, using the monoidal structure of ${_H\Mm^{par}}$, one can introduce a monoidal structure on $\TT({_H\Mm})$ (and on $\DD({_H\Mm})$) such that the restriction functor becomes strictly monoidal. 
Explicitly, for two objects $(M,T_M )$ and $(N,T_N)$ in $\mathbb{T}({}_H \mathcal{M})$, we define
\[
(M,T_M )\bullet (N,T_N )=(M\bullet N, T_{M,N}),
\]
in which
\[
((M\bullet N , T_{M,N}), \varphi_{M,N}) =\left( \left( \overline{T_M (M) \otimes_{A_{par}} T_N (N)}, T_{\pi_{T_M} \otimes \pi_{T_N}} \right) ,\varphi_{T(M)\otimes_{A_{par}} T(N)} \right)
\]
is the standard globalization of the partial $H$-module 
\[
\left( T_M (M) \otimes_{A_{par}} T_N (N) , \pi_{T_M} \otimes \pi_{T_N} \right) \in {}_H \mathcal{M}^{par} .
\]
The monoidal unit is given by the standard globalization $((\overline{{A_{par}}}, T_\alpha ), \varphi_{A_{par}} )$ of the unit object $({A_{par}}, \alpha ) \in {}_H \mathcal{M}^{par}$. 
\end{remark}

In the next proposition we show that global modules can be characterized among partial modules by means of their dilation. 

\begin{prop} \label{globaliso} 
Let $(M, \pi )$ be a partial $H$-module and $((\overline{M}, T_{\pi}), \varphi )$ its standard dilation. Then the following statements are equivalent.
\begin{enumerate}[(i)]
\item $M$ is a (global) $H$-module, that is, $\pi :H \rightarrow \mbox{End}_k (M)$ is a morphism of 
algebras;
\item $\varphi$ has a right inverse partial $H$-module morphism $\ol\varphi:\ol M\to M$;
\item $\varphi$ is bijective. 
\end{enumerate}
In this case $\varphi:M\to \ol M$ is an isomorphism of $H$-modules and  $T_{\pi} =\mbox{Id}_{\overline{M}}$.
\end{prop}

\begin{proof} 
\ul{$(ii)\Leftrightarrow(iii)$}.
Since we $\varphi$ is injective by construction, $\varphi$ is bijective if and only if it has a right inverse. Clearly, if $\varphi$ is bijective, its inverse will also be $H$-linear. 

\ul{$(ii)\Rightarrow(i)$}. Since $\ol M$ is a global $H$-module and $\ol\varphi:\ol M\to M$ is morphism of partial modules, it follows that $M$ is a global $H$-module (see \exref{trivialmod}).

\ul{$(i)\Rightarrow(iii)$}. If $M$ is a global module, then $((M,id),id)$ is a dilation of $M$ which is clearly minimal. Hence by \prref{uniquenessdilation} we find that $\varphi=\Phi:M\to \ol M$ is an isomorphism of $H$-modules and it satisfies $T_\pi\circ \varphi=\varphi$, hence $T_\pi=id$.
\end{proof}

We finish this section by illustrating our results by means of the examples $(k\ZZ_2)^*$ and $H_4$.

\begin{exmp} We consider the Hopf algebra $H=(k\mathbb{Z}_2)^*$ and use the same notation as in \exref{dualC2} and \exref{dualC2_2}; in particular, write $\{p_0,p_1\}$ for the canonical base of $H$, consisting of orthogonal idempotents. Recall that a partial $H$-module (i.e.\ a partially $\ZZ_2$-graded vector space) is a vector space that allows a direct sum decomposition 
\[
\mathbb{V}=\mathbb{V}_0 \oplus \mathbb{V}_1 \oplus \mathbb{V}_{\frac{1}{2}} ,
\]
where the partial representation $\pi:H\to \End(\VV)$ under this ordered direct sum decomposition is presented by the block matrices
\[
\pi(p_0) =\left( \begin{array}{ccc} 1 & 0 & 0 \\ 0 & 0 & 0 \\ 0 & 0 & \frac{1}{2} \end{array} \right) ,
\quad
\pi(p_1) =\left( \begin{array}{ccc} 0 & 0 & 0 \\ 0 & 1 & 0 \\ 0 & 0 & \frac{1}{2} \end{array} \right).
\]
Recall that $H=kp_0\oplus kp_1$. Since $\Hom_k(kp_i,V_j)\cong V_j$, we can identify $\Hom_k (H,\VV) $ with the vector space of $3\times 2$ block matrices, where an $(i,j)$ block with $i \in \{0,1\}$ and $j \in \{0,1,1/2\}$ corresponds to $\Hom(kp_i,V_j)\cong V_j$. The linear map $\varphi:\VV\to \Hom_k(H,\VV),\ \varphi(v)(h)=\pi(h)(v)$ is understood as
\[
\varphi (v_0 )=\left( \begin{array}{cc} v_0 & 0 \\ 0 & 0 \\ 0 & 0 \end{array} \right) ; \;
\varphi (v_1 )=\left( \begin{array}{cc} 0 & 0 \\ 0 & v_1 \\ 0 & 0 \end{array} \right) ; \;
\varphi (v_{1/2} )=\left( \begin{array}{cc} 0 & 0 \\ 0 & 0 \\ \frac{1}{2}v_{1/2} & \frac{1}{2}v_{1/2}  \end{array} \right) ,
\]
for any $v_0\in\VV_0$, $v_1\in \VV_1$ and $v_{1/2}\in\VV_{1/2}$. 
In order to determine the dilated $H$-module $\overline{\mathbb{V}}$, we have to determine the values of $p_i \triangleright \varphi (v_j)$, for $i=0,1$ and $j=0,1,1/2$:
\[
\begin{array}{ccc}
p_0 \triangleright \varphi (v_0 )= \left( \begin{array}{cc} v_0 & 0 \\ 0 & 0 \\ 0 & 0 \end{array} \right)  ; & p_0 \triangleright \varphi (v_1 ) = 
\left( \begin{array}{cc} 0 & 0 \\ 0 & 0 \\ 0 & 0 \end{array} \right)  ; & p_0 \triangleright \varphi (v_{1/2} ) = 
\left( \begin{array}{cc} 0 & 0 \\ 0 & 0 \\ \frac{1}{2}v_{1/2} & 0 \end{array} \right) ; \\
 \, & \, & \, \\
p_1 \triangleright \varphi (v_0 )= \left( \begin{array}{cc} 0 & 0 \\ 0 & 0 \\ 0 & 0 \end{array} \right)   ; & p_1 \triangleright \varphi (v_1 ) = 
\left( \begin{array}{cc} 0 & 0 \\ 0 & v_1 \\ 0 & 0 \end{array} \right)  ; & p_1 \triangleright \varphi (v_{1/2} ) = 
\left( \begin{array}{cc} 0 & 0 \\ 0 & 0 \\ 0 & \frac{1}{2}v_{1/2} \end{array} \right) .
\end{array}
\]
Then, the dilated $H$-module is a subspace of the space of $3\times 2$ block matrices generated by the elements $p_0 \triangleright \varphi (v_0 ), p_0 \triangleright \varphi (v_{1/2} ), p_1 \triangleright \varphi (v_1 )$ and $p_1 \triangleright \varphi (v_{1/2} )$, i.e.
$$\ol\VV = \left(\begin{array}{cc} \VV_0 & 0\\ 0 & \VV_1 \\ \VV_{1/2} & \VV_{1/2}  \end{array}\right)\cong (\VV_0\oplus \VV_{1/2})\oplus (\VV_0\oplus \VV_{1/2}) , $$
which is a $\ZZ_2$-graded vector space by $\ol\VV_0=\VV_0\oplus \VV_{1/2}$ and $\ol\VV_1=\VV_1\oplus \VV_{1/2}$, i.e. the partial $1/2$-degree space is doubled and one copy is placed in each degree. Finally,
for any 
$$\ol v=\left(\begin{array}{cc}v_0 & 0\\ 0 & v_1 \\ v_{1/2} & v'_{1/2}\end{array}\right)\in \ol \VV$$ 
we find that 
$\ol v(1_H)=\ol v(p_0+p_1)=v_0 + v_1 + (v_{1/2} + v'_{1/2})\in \VV$
and therefore 
$$\ol T(\ol v)=\varphi(\ol v(1_H))=
\left(\begin{array}{cc}v_0 & 0\\ 0 & v_1 \\ {1\over 2}(v_{1/2} + v'_{1/2}) & {1\over 2}(v_{1/2} + v'_{1/2})\end{array}\right).$$
Hence the projection map $\ol T:\ol\VV\to \ol\VV$ on the ordered direct sum decomposition $\ol\VV=(\VV_0\oplus \VV_{1/2})\oplus (\VV_0\oplus \VV_{1/2})$ is then described by a block matrix
\[
\overline{T} =\left( \begin{array}{cccc} 1 & 0 & 0 & 0 \\ 0 & 1\over 2 & 0 & 1\over 2 \\
0 & 0 & 1 & 0 \\ 0 & 1\over 2 & 0 & 1\over 2 \end{array} \right) .
\]
Then the projection $\overline{T}$ coincides with the projection of the Example 
\ref{Z2grading}.
\end{exmp}

\begin{exmp} Consider the Sweedler Hopf algebra $H_4$ and consider a generic $n$-dimensional pure partial $H_4$-module $W$, as described in \exref{exampleSweedler}. I.e. we have matrices $c,d\in\End(W)$ satisfying $cd=dc$ and $c^2=d^2$ which define a partial $H_4$-module by means of the action
$$\pi(g)(w)=0, \quad \pi(x)(w)=cw, \quad \pi(y)(w)=dw.$$
To construct the dilation of $W$, consider $\Hom_k(H_4,W)\cong W^4$, using the canonical base $\{1,g,x,y=gx\}$ for $H_4$. Then we find that
$$\varphi:W\to \Hom_k(H_4,W), \varphi(w)=(w,0,cw,dw),$$
and $\ol{W}=H\triangleright \varphi(W)$ is spanned by the elements of the form
\begin{eqnarray*}
\varphi(w)&=&(w,0,cw,dw), \\
g\triangleright\varphi(w)&=&(0,w,-dw,-cw), \\
x\triangleright\varphi(w)&=&(cw,dw,0,0), \\
y\triangleright\varphi(w)&=&(dw,cw,0,0).
\end{eqnarray*}
Then obviously $x\triangleright\varphi(w)$ and $y\triangleright\varphi(w)$ are linear combinations of elements of the form $\varphi(w)$ and $g\triangleright\varphi(w)$.  Hence we obtain a $k$-linear isomorphism $\alpha:\ol W \to W\oplus W$, given by $\alpha(\varphi(w))=(w,0)$ and $\alpha(g\triangleright \varphi(w))=(0,w)$. By transfer of structure, we find that $W\oplus W$ can be endowed with a (global) $H_4$-action is then given by
$$g\triangleright (w,w')=(w',w),\quad x\triangleright (w,w')=(cw-dw',dw-cw'),$$
such that the map $\alpha$ above becomes an isomorphism of (global) $H_4$-modules.
In other words, we obtain exactly a global $H_4$-module of even dimension as described in \exref{Sweedler2}, whose restricted partial action is the pure submodule $W$ we started with.
\end{exmp}

\section{Dilations and globalizations}

\subsection{Exactness of the dilation functor}

\begin{lemma}
Let $H$ be a Hopf algebra.
\begin{enumerate}[(i)]
\item $\ol{H_{par}}$ is an $H$-$H_{par}$ bimodule via the actions 
$$h\triangleright (h_i\triangleright \varphi(x_i))\triangleleft y=hh_i\triangleright \varphi(x_iy),$$ 
for all $h,h_i\in H$ and $x_i,y\in H_{par}$;
\item Let $M$ be a partial $H$-module, then by part (i) $\ol{H_{par}}\ot_{H_{par}}M$ inherits a left $H$-module structure from $\ol{H_{par}}$. Define
\begin{eqnarray*}
T:\ol{H_{par}}\ot_{H_{par}}M\to \ol{H_{par}}\ot_{H_{par}}M,&& T(f\ot m)=\varphi(1)\ot f(1)\triangleright m,\\ 
\phi:M\to \ol{H_{par}}\ot_{H_{par}}M,&& \phi(m)=\varphi(1)\ot m.
\end{eqnarray*} 
Then $((\ol{H_{par}}\ot_{H_{par}}M,T),\phi)$ is a proper dilation of $M$. We call this the {\em secondary dilation} of $M$.
\end{enumerate}
\end{lemma}

\begin{proof}
\ul{(i)}.
Since we know that the dilation functor $D:{_H\Mm^{par}}\to {_H\Mm}$ is additive (see \prref{propdilationfunctor}) and ${_H\Mm^{par}}={_{H_{par}}\Mm}$, it follows that $D(H_{par})=\ol{H_{par}}$ is an $H$-$H_{par}$ bimodule with the given structure.\\
\ul{(ii)}.
First remark that for any $f=h_i\triangleright \varphi(x_i)\in \ol{H_{par}}$ we have that $f(1)=h_i\cdot x_i=[h_i]x_i$. Hence, for all $k\in H$ we have 
$$\varphi([h_{(1)}][S(h_{(2)})])(k)=[k][h_{(1)}][S(h_{(2)})]=[kh_{(1)}][S(h_{(2)})]
=(h_{(1)}\triangleright \varphi([S(h_{(2)})]))(k),$$
i.e. $\varphi([h_{(1)}][S(h_{(2)})])=h_{(1)}\triangleright \varphi([S(h_{(2)})])$. 
Let us check that $T$ satisfies the $c$-condition.
\begin{eqnarray*}
h_{(1)}\triangleright T\big(S(h_{(2)})\triangleright T(f\ot_{H_{par}} m)\big)&=&
h_{(1)}\triangleright T\big(S(h_{(2)})\triangleright \varphi(1)\ot_{H_{par}} f(1)\triangleright m\big)\\
&=& h_{(1)}\triangleright \varphi(1)\ot_{H_{par}}  [S(h_{(2)})]f(1)\triangleright m\\
&=& h_{(1)}\triangleright \varphi([S(h_{(2)})])\ot_{H_{par}} f(1)\triangleright m ,\\
T\big(h_{(1)}\triangleright T (S(h_{(2)})\triangleright f\ot_{H_{par}} m\big)
&=& T\big(h_{(1)}\triangleright \varphi(1)\ot_{H_{par}} [S(h_{(2)})]f(1)\triangleright m\big)\\
&=& \varphi(1)\ot_{H_{par}} [h_{(1)}][S(h_{(2)})]f(1)\triangleright m\\
&=& \varphi([h_{(1)}][S(h_{(2)})])\ot_{H_{par}} f(1)\triangleright m .
\end{eqnarray*}
Clearly, the image of $\phi$ equals the image of $T$ and $\phi$ is injective. Let us check that $\phi$ is a morphism of partial modules. 
\begin{eqnarray*}
\phi(h\cdot m)&=&\varphi(1)\ot h\cdot m = \varphi(1)\ot [h]m\\
&=&T(h\triangleright \varphi(1)\ot m)=T(h\triangleright \phi(m))\\
&=& h\cdot \phi(m).
\end{eqnarray*}
Hence $\phi$ induces an isomorphism of partial modules $\phi:M\to T(\ol{H_{par}}\ot_{H_{par}}M)$. 
Finally, let us verify that $\ol{H_{par}}\ot_{H_{par}}M$ is generated by the image of $\phi$ as left $H$-module.
$$ \sum_i h_i\triangleright \varphi(x_i)\ot_{H_{par}} m = \sum_i  h_i\triangleright\varphi(1)\ot_{H_{par}} x_im= \sum_i  h_i\triangleright\phi(x_im).$$
We can conclude that the secondary dilation is well-defined and proper.
\end{proof}

\begin{proposition}
Let $H$ be a Hopf algebra. The following statements are equivalent.
\begin{enumerate}
\item the dilation functor is exact;
\item the dilation functor preserves kernels;
\item the dilation functor has a right adjoint;
\item the dilation functor is naturally isomorphic to the functor $\ol{H_{par}}\ot_{H_{par}}-$;
\item for any partial module, the canonical $H$-module morphism between the secondary and standard dilations is bijective;
\item for any partial module, the secondary dilation is minimal.
\end{enumerate}
Under these equivalent conditions, $\ol{H_{par}}$ is flat as right $H_{par}$-module.
\end{proposition}

\begin{proof}
The equivalence of the first four items follows form the Eilenberg-Watts theorem allong with the properties proven in \prref{propdilationfunctor}. 
The equivalence between the three last items follow from \prref{uniquenessdilation}. The last statement follows from the fact that under these equivalent conditions, the functor $\ol{H_{par}}\ot_{H_{par}}-$ is exact.
\end{proof}

\begin{remark}
Unfortunately, we were not able to prove whether the equivalent conditions in the previous proposition are satisfied for any Hopf algebra $H$. Nevertheless, we also didn't find a counter example. It follows from the explicit description of dilations given above of both the two-dimensional Hopf algebra $(k\Z_2)^*$ and Sweedler's four dimensional Hopf algebra $H_4$, that in these examples the dilation functor preserves kernels and hence is exact.

In case the equivalent conditions of the previous proposition are satisfied, it follows that to understand all dilations, it is enough to understand the dilation of the regular partial module $H_{par}$. Since we know moreover that $H_{par}\cong \ul{{A_{par}}\# H}$ (see \seref{globpart}), the previous question reduces furthermore to understand the dilation of a partial smash product. This we will do in the following two subsections.
\end{remark}

\subsection{The dilations of a partial smash product}

Recall that $B$ is an algebra in the monoidal category of partial $H$-modules if and only if the partial $H$-module structure $H\ot B\to B$ defines a partial action of $H$ on $B$. 
We know from \cite{AB} that given any symmetric partial action of $H$ on $B$, there exists a {\em enveloping action} or {\em globalization}, which consists of a (possibly non-unital) $H$-module algebra $\ol B$, together with an injective multiplicative map $\theta:B\to \ol B$ such that the image $\theta$ is an ideal in $\ol B$ that generates $\ol B$ as an $H$-module and the partial action on $B$ can be obtained by restricting the (global) action on $\ol B$.

The next proposition relates the globalization of a partial action with the dilation of its associated partial $H$-module. Recall that a (non-unital) ring $R$ said to be {\em idempotent} if $RR=R$, i.e. for any $r\in R$ there exist (non-unique) elements $r_i,r'_i\in R$ such that $r=\sum_i r_ir'_i$.

\begin{prop}\prlabel{AB} Let $\cdot : H\otimes B \rightarrow B$ be a symmetric partial action of the Hopf algebra $H$ on a unital algebra $B$ and let $(B, \pi )$ be its 
associated $H$-module structure, given by the partial representation $\pi (h)(a)=h\cdot a$. Consider the standard dilation 
$((\overline{B} ,T_{\pi}),\varphi )$ of $(B,\pi)$. Then
\begin{enumerate}
\item $\overline{B}$ is an idempotent (but not necessarily unital) $H$-module algebra with the convolution product of $\Hom_k (H,B)$;
\item the map $\varphi :B\rightarrow \overline{B}$ is multiplicative;
\item $\varphi (B)$ is an ideal in $\overline{B}$.
\end{enumerate}
Consequently, $\ol B$ is exactly the globalization of the partial action on $B$. 
\end{prop}

\begin{proof} 
The statements (i)-(iii) can be easily verified. In particular, to see that $\ol B$ is idempotent, take any $\sum_i  h_i\triangleright \varphi(b_i) \in \ol B$. Then 
\begin{eqnarray*}
\sum_i \big((h_{i(1)}\triangleright\varphi(b_i))*(h_{i(2)}\triangleright \varphi(1))\big)(k)
&=& \sum_i  (h_{i(1)}\triangleright\varphi(b_i))(k_{(1)})(h_{i(2)}\triangleright \varphi(1))(k_{(2)})\\
&=& \sum_i  (k_{(1)}h_{i(1)}\cdot b_i)(k_{(2)}h_{i(2)}\cdot 1)= \sum_i  kh_i\cdot b_i\\
&=& \big( \sum_i  h_i\triangleright \varphi(b_i)\big)(k).
\end{eqnarray*}
The remaining results follow by \cite{AB}.
\end{proof}

\begin{lemma}\lelabel{parglobtensor}
If $M$ is a partial $H$-module and $N$ is a global $H$-module then the tensor product $M\ot N$ is a partial $H$-module by means of the diagonal action
$$\cdot: H\ot M\ot N\to M\ot N,\ h\cdot (m\ot n)=h_{(1)}\cdot m\ot h_{(2)}n$$
\end{lemma}

\begin{proof}
For all $m\ot n\in M$ and $h,k\in H$, we find
\begin{eqnarray*}
k\cdot (S(h_{(1)})\cdot (h_{(2)}\cdot (m\ot n)))&=& k_{(1)}\cdot (S(h_{(2)})\cdot (h_{(3)}\cdot m))\ot k_{(2)}S(h_{(1)})h_{(4)}n\\
&=& k_{(1)}S(h_{(2)})\cdot (h_{(3)}\cdot m)\ot k_{(2)}S(h_{(1)})h_{(4)}n\\
&=& kS(h_{(1)})\cdot (h_{(2)}\cdot (m\ot n))
\end{eqnarray*}
and
\begin{eqnarray*}
k\cdot (h_{(1)}\cdot (S(h_{(2)})\cdot (m\ot n)))&=& k_{(1)}\cdot (h_{(1)}\cdot (S(h_{(4)})\cdot m))\ot k_{(2)}h_{(2)}S(h_{(3)})n\\
&=& k_{(1)}\cdot (h_{(1)}\cdot (S(h_{(2)})\cdot m))\ot k_{(2)}n\\
&=& k_{(1)}h_{(1)}\cdot (S(h_{(2)})\cdot m)\ot k_{(2)}n\\
&=& k_{(1)}h_{(1)}\cdot (S(h_{(4)})\cdot m)\ot k_{(2)}h_{(2)}S(h_{(3)})n\\
&=& kh_{(1)}\cdot (S(h_{(2)})\cdot (m\ot n)) .
\end{eqnarray*}
This shows two of the four associativity axioms, the other two are proven in the same way.
\end{proof}

Recall that the partial smash product $\ul{B\# H}$ is the direct summand of $B\ot H$ consisting of elements of the form
$$b\# h=b(h_{(1)}\cdot 1)\ot h_{(2)},$$
and it becomes a unital algebra with unit $1_B\# 1_H$ and product
$$(a\# h)(b\# k)=a(h_{(1)}\cdot b)\# h_{(2)}k.$$
One can now easily observe that $\ul{B\# H}$ is a direct summand of $B\ot H$ as partial left $H$-modules.

\begin{lemma}
Consider a partial $H$-module algebra $B$ with globalization $\ol B$.
The morphism
$$\Hom(H,B)\ot H\to \Hom(H,B\ot H), f \ot h \mapsto (k\mapsto f(k)\ot h)$$
is injective. Hence an element $\sum_i  h_i\triangleright \varphi(b_i)\ot h\in \ol B\ot H$ is zero if and only if for all $k\in H$
$$\sum_i  kh_i\cdot b_i\ot h=0.$$
\end{lemma}

\begin{proof}
Write $\sum_i f_i\ot h_i\in \Hom(H,B)\ot H$ with $h_i$ linearly independent. Suppose that 
$\sum_i  f_i(k)\ot h_i=0$ for all $k\in H$. Consider any functional $\alpha\in B^*$. Then we find that $\sum_i  \alpha(f_i(k))h_i=0$ and since the $h_i$ are linearly independent it follows that $\alpha(f_i(k))=0$ for all $k$, $i$ and $\alpha$. Hence we find that $f_i(k)=0$ for all $i$ and $k$, therefore $f_i=0$ for all $i$. Hence $\sum_i  f_i\ot h_i=0$ and the map is injective as stated.

The second part follows since $\ol B\subset \Hom(H,B)$.
\end{proof}

\begin{theorem}\thlabel{globalsmash}
Let $\cdot : H\otimes B \rightarrow B$ be a partial action of the Hopf algebra $H$ on a unital algebra $B$ and let $\ol B$ be the globalization of $B$. Then the dilation of the partial smash product $\ul{B\# H}$ is a direct summand, as an $H$-module, of the (global) smash product $\ol B\# H$. 
\end{theorem}

\begin{proof}
By \leref{parglobtensor}, $B\ot H$ is a partial $H$-module. Let us show that $\ol{B\ot H}\cong \ol B\ot H$

We define the maps
\begin{eqnarray}
\zeta:\ol{B\ot H}\to {\ol B}\ot H,&\quad& \zeta(\sum_i h_i\triangleright \varphi(b^i\ot h^i))=\sum_i h_{i(1)}\triangleright \varphi(b^i)\ot h_{i(2)}h^i. \eqlabel{zeta}\\
\xi: {\ol B}\ot H\to \ol{B\ot H},&\quad& \xi( (\sum_i h_i\triangleright \varphi(b_i))\ot h)=\sum_i h_{i(1)}\triangleright \varphi(b_i\ot S(h_{i(2)})h).\nonumber
\end{eqnarray}
Let us first check that $\zeta$ is well-defined. Suppose that $\sum_i h_i\triangleright \varphi(b^i\ot h^i)=0$, i.e. for all $k$, we have 
$$\sum_i  k_{(1)}h_{i(1)}\cdot b^i\ot k_{(2)}h_{i(2)}h^i=0 .$$
Then we find for all $k$,
\begin{eqnarray*}
\sum_i kh_{i(1)}\cdot b^i\ot h_{i(2)}h^i&=&\sum_i k_{(1)}h_{i(1)}\cdot b^i\ot S^{-1}(k_{(3)})k_{(2)}h_{i(2)}h^i=0 ,
\end{eqnarray*}
hence
$$\sum_i kh_{i(1)}\cdot b^i\ot h_{i(2)}h^i=0.$$
By the previous lemma, we then find that $\sum_i h_{i(1)}\triangleright \varphi(b^i)\ot h_{i(2)}h^i=0$.

By a similar argument, we find that $\xi$ is well-defined. First remark that $\ol B\ot H \subset \Hom(H,B)\subset \Hom(H,B\ot H)$. Hence 
$(\sum_i h_i\triangleright \varphi(b_i))\ot h=0$
in $\ol B \ot H$ if and only if for all $k\in H$, we have 
$$(\sum_i kh_i\cdot b_i)\ot h=
0$$
in $B\ot H$. But this implies that for all $k\in H$ 
\begin{eqnarray*}
\sum_i k_{(1)}h_{i(1)}\cdot b_i\ot k_{(2)}h_{i(2)}S(h_{i(3)})h&=&
\sum_i k_{(1)}h_{i}\cdot b_i\ot k_{(2)}h
=0,
\end{eqnarray*}
which means exactly that 
$$\xi((\sum_i kh_i\cdot b_i)\ot h)=\sum_i h_{i(1)}\triangleright \varphi(b_i\ot S(h_{i(2)})h)=0$$ 
in $\ol{B\ot H}$.

Let us check that $\zeta$ and $\xi$ are mutual inverses:
\begin{eqnarray*}
\xi\circ \zeta(\sum_i h_i\triangleright \varphi(b^i\ot h^i)) &= & \xi(\sum_i h_{i(1)}\triangleright \varphi(b^i)\ot h_{i(2)}h^i)\\
&=& \sum_i h_{i(1)}\triangleright \varphi(b^i \ot S(h_{i(2)})h_{i(3)}h^i) \\
&=& \sum_i h_{i}\triangleright \varphi(b^i \ot h^i)
\end{eqnarray*}
and 
\begin{eqnarray*}
\zeta\circ \xi((\sum_i h_i\triangleright \varphi(b_i))\ot h) &=& \zeta(\sum_i h_{i(1)}\triangleright \varphi(b_i\ot S(h_{i(2)})h) )\\
&=& \sum_i (h_{i(1)}\triangleright \varphi(b_i) ) \ot h_{i(2)}S(h_{i(3)}) h \\
&=& \sum_i (h_{i}\triangleright \varphi(b_i) ) \ot h .
\end{eqnarray*}

Moreover $\ul{B\# H}$ is a direct summand of $B\ot H$ as partial $H$-module, 
it follows that $\ol{\ul{B\# H}}$ is a direct summand of $\ol{B\ot H}\cong \ol B\ot H$. Since $\ol B\#H$ is just $\ol B\ot H$ as left $H$-module we obtain the stated result.
\end{proof}

\subsection{The Morita equivalence between partial and global smash products}

Let $B$ be a partial $H$-module algebra and consider the partial $H$-module $B\ot H$ as in \leref{parglobtensor}. Combining the inclusion map $\varphi:B\ot H \to \ol{B\ot H}$ from the standard dilation with the split monomorphism $\zeta:\ol{B\ot H}\to {\ol B}\ot H$ from \thref{globalsmash}, we obtain an injective map
$\Phi :B\ot H\rightarrow \overline{B}\ot H$ given by
\[
\Phi (b\ot h)=\varphi_B (b) \ot h.
\]
Restricting this map to the direct summand $\ul{B\# H}$ of $B\ot H$, we obtain 
$\Phi :\underline{B\# H}\rightarrow \overline{B}\# H$ given by
\[
\Phi (\ul{b\# h})=\varphi_B (b(h_{(1)}\cdot 1)) \# h_{(2)} 
= \varphi_B (b)(h_{(1)}\triangleright \varphi_B(1) ) \# h_{(2)} 
= \varphi_B (b(h_{(1)}\cdot 1))(h_{(2)}\triangleright \varphi_B (1) ) \# h_{(3)} .
\]
To see that the three expressions for $\Phi (\ul{b\# h})$ are indeed equal, consider any $b\in B$ and $h,k\in H$, then we find
\begin{eqnarray*}
\varphi_B(b(h\cdot 1))(k)
&=& k\cdot (b(h\cdot 1))\\
&=& (k_{(1)}\cdot b)(k_{(2)}\cdot (h\cdot 1))\\
&=& (k_{(1)}\cdot b)(k_{(2)}\cdot 1)(k_{(3)}h\cdot 1)\\
=\big(\varphi_B(b)(h\triangleright \varphi(1))\big)(k)&=& (k_{(1)}\cdot b)(k_{(2)}h\cdot 1)\\
&=& (k_{(1)}\cdot b) (k_{(2)}h_{(1)}\cdot 1)(k_{(3)}h_{(2)}\cdot 1)\\
&=& (k_{(1)}\cdot b) (k_{(2)}\cdot 1)(k_{(3)}h_{(1)}\cdot 1)(k_{(4)}h_{(2)}\cdot 1)\\
&=& (k_{(1)}\cdot b)(k_{(2)}\cdot (h_{(1)}\cdot 1))(k_{(3)}h_{(2)}\cdot 1)\\
=\varphi_B (b(h_{(1)}\cdot 1))(h_{(2)}\triangleright \varphi (1))(k) &=&
(k_{(1)}\cdot (b(h_{(1)}\cdot 1)))(k_{(2)}h_{(2)}\cdot 1)
\end{eqnarray*}
Moreover, one can easily check that $\Phi$ is an algebra map. In \cite{AB}, a Morita context with surjective Morita maps was constructed between the partial smash product $\underline{B\# H}$ and the usual smash product $\overline{B}\# H$. More explicitly, one has an $\underline{B\# H}$-$\overline{B}\# H$ bimodule $P$ and a $\overline{B}\# H$-$\underline{B\# H}$ bimodule $Q$ given by
\begin{eqnarray*}
P &=& \Phi (B\otimes H)=\left\{ \sum_{i=1}^n \varphi_B (a_i )\otimes h_i \; | \; a_i \in B, \; h_i \in H, \; n\in \mathbb{N} \right\} \subseteq \overline{B}\# H, \\
Q &=& \left\{ \sum_{i=1}^n h^i_{(1)} \triangleright \varphi_B (a_i) \otimes h^i_{(2)} \; | \; a_i \in B, \; h_i \in H, \; n\in \mathbb{N} \right\} \subseteq \overline{B}\# H ,
\end{eqnarray*}
where the bimodule actions and Morita maps 
\begin{eqnarray*}
\tau:P\ot_{\ol B\# H}Q&\to& \ul{B\# H}\\
\mu:Q\ot_{\ul{B\# H}}P &\to& \ol B\# H
\end{eqnarray*}
are all defined by means of the multiplication in $\overline{B}\# H$, after identifying $\ul{B\# H}$, $P$ and $Q$ with subspaces of $\overline{B}\# H$. Since $\varphi_B (B)$ is an ideal in $\overline{B}$, one can easily check that all actions and the Morita maps are well-defined. Moreover, as proven in \cite[Theorem 4]{AB}, the Morita maps are surjective. 
Since in general, $\ol B$ and hence $\ol B\# H$ are not unital algebras, the surjectivity of $\mu$ does not necessarily imply its bijectivity (as it is the case for $\tau$, by classical Morita theory). However, 
we can use the results from \cite{BV2} to conclude on the following result that generalizes \cite{AB}. Recall that a {\em firm} (right) module over a non-unital ring $R$ is a (right) module $M$ such that the multiplication map induces an isomorphism $M\ot_R R\cong M$. A ring $R$ is called firm if it it firm as a (left or right) module over itself. If $R$ has a unit, then firm modules over $R$ are exactly the (unital) modules. 

The terminology used above is due to Quillen (see e.g. \cite{BV2}). In literature, other names are in use as well (e.g.\ `unital module' for firm module or `tensor idempotent ring' for firm ring), however we believe that this terminology causes the least confusion.

\begin{proposition}
Let $B$ be a partial $H$-module algebra. Then 
\begin{enumerate}[(i)]
\item the categories of (unital) $\ul{B\# H}$-modules and of firm $\overline{B}\# H$-modules are equivalent;
\item the above categories are moreover equivalent with the categories of firm modules over the firm ring $\overline{B}\# H\ot_{\overline{B}\# H}\overline{B}\# H$.
\end{enumerate}
\end{proposition}

\begin{proof}
\ul{(i)}. 
We already know that the Morita map $\mu$ is bijective. Thanks to the Kato-Osake theorem, in the form of  \cite[Lemma 1.10]{BV2}, we also know that $\tau$ induces a bijective morphism 
$$id\ot \mu: M\ot_{\overline{B}\# H}Q\ot_{\ul{B\# H}}P \to M\ot_{\overline{B}\# H}\ol B\# H\cong M$$
for any firm $\ol B\#H$-module $M$. Hence we obtain the demanded equivalence of categories.
\\
\ul{(ii)}. We know that $\ol B$ is an idempotent ring (see \prref{AB}), hence this follows from \cite[Theorem 1.1]{BV2}.
\end{proof}

Finally, let us consider the case where $B={A_{par}}$ is the partial $H$-module algebra satisfying $H_{par}=\ul{{A_{par}}\# H}$. Since we know that $H_{par}$ is a Hopf algebroid, its category of modules is monoidal. Hence, we can use the Morita equivalence described above, to obtain a monoidal structure on the category of firm $\ol {A_{par}}\#H$-modules. Although in general, we don't know if $\ol {A_{par}}\#H$ has a Hopf algebroid structure, we can use the monoidal structure on its module category to endow it with a {\em Hopfish} structure, in the following sense.

\begin{defi}  A $k$-algebra $\mathcal{A}$ is said to be a sesqui-unital sesqui-algebra (see \cite{TWZ}) if there are two bimodules 
$\nabla \in {}_{\mathcal{A}} \mathcal{M}_{\mathcal{A}\otimes \mathcal{A}}$ and $E \in {}_{\mathcal{A}} \mathcal{M}_k$ satisfying
\begin{enumerate}
\item[(i)] $ \nabla \otimes_{\mathcal{A}\otimes \mathcal{A}} (\nabla \otimes \mathcal{A}) \cong 
\nabla \otimes_{\mathcal{A} \otimes \mathcal{A}} (\mathcal{A}\otimes \nabla )$;
\item[(ii)] $\nabla \otimes_{\mathcal{A}\otimes \mathcal{A}} (E \otimes \mathcal{A})\cong \mathcal{A} \cong \nabla \otimes_{\mathcal{A} \otimes \mathcal{A}} (\mathcal{A} \otimes E )$.
\end{enumerate}
These isomorphism moreover satisfy appropriate coherence axioms, that make $\Aa$ into a pseudo comonad in the bicategory of algebras, bimodules and bimodule morphisms.
\end{defi}

It is known \cite{dCV,V} that a (unital) algebra is a sesqui-unital sesqui-algebra if and only if its category of modules is closed monoidal (without assumptions on the existence of a monoidal fibre functor). 
In \cite{TWZ}, a notion of antipode for sesqui-unital sesqui-algebras was proposed, which was however rather rigid. 
In \cite{dCV}, an alternative notion of antipode is developed and it is shown that this new notion is invariant under Morita equivalence. For sake of completeness, we recall that notion here.

\begin{defi}
A sesqui-unital sesqui-algebra $\mathcal{A}$ is called a {\em Hopfish} algebra (see \cite{dCV}) if there is a $\mathcal{A}$-$\mathcal{A}^{op}$ bimodule $\Sigma$ that is finitely generated and projective as a right $\mathcal{A}^{op}$ module and the category of (left) $\Aa$-modules that are finitely generated and projective is rigid with duals given by $\Sigma\ot_{\Aa^{op}}X^*$. We  call $\Sigma$ the antipode of $\Aa$.
\end{defi}

In fact, a $\Sigma$ is an antipode of $\Aa$ if and only if there is a canonical isomorphisms of $\Aa$-$\Aa^{op}$-bimodules
\[
\underline{\mbox{can}} :\nabla \otimes_{\mathcal{A}\otimes \mathcal{A}} (A\otimes \Sigma) \rightarrow {}{\Hom_\mathcal{A}} (\nabla_1 , \mathcal{A})
\]
where $\nabla_1=\nabla$ as $k$-module and endowed with a right $\Aa$-module structure 
$$x\cdot a=x(a\ot 1)$$
for any $x\in\nabla_1=\nabla$. The two remaining $\Aa$-module structures of $\nabla$ are used to endow ${\Hom_\mathcal{A}} (\nabla_1 , \mathcal{A})$ with an $\Aa$-$\Aa^{op}$ bimodule structure.

If $\Hh$ is a Hopf algebroid with base $k$-algebra $A$, then it can endow it with the structure of a Hopfish $k$-algebra as follows:
\begin{itemize}
\item the comultiplication bimodule is $\Hh\ot\Hh$ with the regular right $\Hh\ot\Hh$-action and the left action given by
$$a\cdot (b\ot b')=\Delta(a)(b\ot b');$$
\item The counit is the base algebra $A$ with left $\Hh$-action $a\cdot 1_A=\epsilon(a);$
\item The antipode $\Hh$-$\Hh^{op}$ bimodule is $\Hh^{op}$ with regular right $\Hh^{op}$ action and left $\Hh$ action induced by the antipode map $S:\Hh\to \Hh^{op}$.
\end{itemize}

At least in the case when $\ol {A_{par}}\# H$ is a unital algebra (e.g. $H$ is a finite group algebra), then it follows from the discussion above that $\Aa=\ol {A_{par}}\# H$, being Morita equivalent to the Hopf algebroid $H_{par}$, is itself a Hopfish algebra, whose structure is given by
\begin{enumerate}
\item The comultiplication $\Aa$-$\Aa\ot \Aa$ bimodule  $\nabla=Q\otimes_{H_{par}} (P\otimes _{A_{par}} P);$
\item The counit $k$-$\Aa$ bimodule 
$
E = Q\otimes_{H_{par}} {A_{par}} ;
$
\item The antipode $\Aa$-$\Aa^{op}$ bimodule $\Sigma$ is the quotient of the $k$-module $Q\ot Q$ by its subspace generated by elements of the form
$$qx\ot q'-q\ot q'S(x),$$
where $q,q'\in Q$ and $x\in H_{par}$. The $\Aa$-$\Aa^{op}$ bimodule structure on $\Sigma$ is given by
$$a\cdot (q\ot q')\cdot b = aq\ot bq'$$
for all $a,b\in\Aa$.
\end{enumerate}

\subsection*{Acknowledgement}
{The first author is partially supported by CNPq, grant number 306583/2016-0.
The third author would like to thank the FWB (F\'ed\'eration Wallonie-Bruxelles) for financial support via the ARC-project ``Hopf algebras and the symmetries of non-commutative spaces'' as well as the FNRS for the MIS-project ``Antipode''.\\
All authors thank the anonymous referee for his/her useful comments that improved the presentation of this paper.
}

\end{document}